\documentclass[11pt]{amsart}

\usepackage{amssymb}
\usepackage{mathrsfs} 

\usepackage[latin1]{inputenc}
\usepackage[colorlinks=true, citecolor=blue]{hyperref} 

\allowdisplaybreaks[1] 

\newtheoremstyle{fancy}{}{}{\itshape}{}{\textbf\bgroup}{.\egroup}{ }{}
\newtheoremstyle{fanci}{}{}{\rm}{}{\textbf\bgroup}{.\egroup}{ }{}
\newtheoremstyle{ghost}{}{}{\itshape}{}{\textbf\bgroup}{\egroup}{ }{}

\theoremstyle{fancy}
\numberwithin{equation}{section} \swapnumbers
\newtheorem{cor}[equation]{Corollary}
\newtheorem{lem}[equation]{Lemma}
\newtheorem{prop}[equation]{Proposition}
\newtheorem{thm}[equation]{Theorem}

\newtheorem*{Thurston}{The Thurston Geometries}

\newtheorem*{MetricallyMaximal}{The Metrically Maximal Three-Dimensional Geometries}
\newtheorem*{ConstantCurvatureGeometry}{The Geometries of Constant Sectional Curvature}

\theoremstyle{fanci}

\newtheorem{dfn}[equation]{Definition}

\newtheorem{rem}[equation]{Remark}

\newtheorem*{prob}{Problem}
\newtheorem*{Structure}{Structure of the Paper}
\newtheorem*{Acknowledgments}{Acknowledgments}

\newcommand{\cref}[1]{Corollary~\ref{#1}}   

\newcommand{\R}{\mathbb R} 
\newcommand{\N}{\mathbb N} 

\newcommand{\SO}{\operatorname{SO}}

\newcommand{\Ad}{\operatorname{Ad}}

\newcommand{\Inner}{\langle \cdot , \cdot \rangle}

\newcommand{\Isom}{\operatorname{Isom}}

\newcommand{\Ric}{\operatorname{Ric}}
\newcommand{\Scal}{\operatorname{Scal}}
\newcommand{\vol}{\operatorname{vol}}

\newcommand{\germ}{\mathfrak}

\renewcommand{\phi}{\varphi}

\begin{document}

\newcommand{\spacing}[1]{\renewcommand{\baselinestretch}{#1}\large\normalsize}
\spacing{1.14}

\title{Geometric Structures and the Laplace Spectrum}

\author[S. Lin]{Samuel Lin}
\address{Dartmouth College\\ Department of Mathematics \\ Hanover, NH 03755}
\email{samuel.lin@dartmouth.edu}
\author[B. Schmidt]{Benjamin Schmidt}
\address{Michigan State University\\ Department of Mathematics \\ E. Lansing, MI 48824}
\email{schmidt@math.msu.edu}
\author[C. Sutton]{Craig Sutton$^\sharp$}
\address{Dartmouth College\\ Department of Mathematics \\ Hanover, NH 03755}
\email{craig.j.sutton@dartmouth.edu}
\thanks{$^\sharp$ The third named author was partially supported by a Simons Foundation Collaboration Grant}

\subjclass[2010]{58J50, 53C20}
\keywords{Laplace spectrum, heat invariants, geometric structures, three-manifolds}



\begin{abstract}
Inspired by the role geometric structures play in our understanding of surfaces and three-manifolds, and Berger's observation that a surface of constant sectional curvature is determined up to local isometry by its Laplace spectrum, we explore the extent to which compact locally homogeneous three-manifolds are characterized up to local isometry by their spectra. We observe that there are eight ``metrically maximal'' three-dimensional geometries on which all compact locally homogeneous three-manifolds are modeled and we demonstrate that for five of these geometries the associated compact locally homogeneous three-manifolds are determined up to local isometry by their spectra within the universe of locally homogeneous three-manifolds. Specifically, we show that among compact locally homogeneous three-manifolds, a Riemannian three-manifold is determined up to local isometry if its universal Riemannian cover is isometric to (1) a symmetric space, (2) $\R^2 \rtimes \R$ endowed with a left-invariant metric, (3) $\operatorname{Nil}$ endowed with a left-invariant metric, or (4) $S^3$ endowed with a left-invariant metric sufficiently close to a metric of constant sectional curvature. We then deduce that three-dimensional Riemannian nilmanifolds and locally symmetric spaces with universal Riemannian cover $\mathbb{S}^2 \times \mathbb{E}$ are uniquely characterized by their spectra among compact locally homogeneous three-manifolds. Finally, within the collection of closed manifolds covered by $\operatorname{Sol}$ equipped with a left-invariant metric, local geometry is ``audible.'' 
\end{abstract}

\maketitle



\setcounter{section}{0}

\section{\bf Introduction}\label{sec:introduction}

Spectral geometry is the study of the relationship between the spectrum of a Riemannian manifold---i.e., the sequence of eigenvalues (counting multiplicities) of the associated Laplace-Beltrami operator---and its underlying geometry. Two Riemannian manifolds are said to be \emph{isospectral} if their spectra agree and, building off of Kac's metaphor \cite{Kac}, a geometric property will be called \emph{audible}, if it is encoded in the spectrum. Numerous examples of isospectral, yet non-isometric spaces demonstrate that, in general, the spectrum does not completely determine the geometry of the underlying Riemannian manifold. Nevertheless, it is expected that certain natural classes of Riemannian manifolds are characterized by their spectra. 

For example, it is widely believed that a round $n$-sphere is determined up to isometry by its spectrum. In 1973, Tanno verified this for round spheres of dimension at most six \cite[Thm. B]{Tanno73}. Seven years later, he also proved that a round metric on an arbitrary $n$-sphere is locally audible \cite{Tanno80}; that is, within the space of metrics on the $n$-sphere, each metric of constant positive curvature admits a neighborhood in which it is determined up to isometry by its spectrum. Surprisingly, no further progress has been made regarding high-dimensional round spheres. More generally, one expects the spectrum to encode whether a closed Riemannian manifold has constant sectional curvature $K$; however, this is only known to be true in dimension five and lower (cf. \cite{Berger} and \cite[Thm. A]{Tanno73}). 

Given the important role geometric structures (i.e., complete locally homogeneous Riemannian metrics) play in our understanding of surfaces (via the uniformization theorem) and three-manifolds (via the geometrization conjecture) this article explores the extent to which low-dimensional geometric structures are audible. 

\subsection{Geometric Structures} 
By an $n$-dimensional \emph{geometry} we shall mean a triple $(X,G, \alpha)$ consisting of a smooth simply-connected $n$-dimensional manifold $X$, a connected Lie group $G$, and a smooth transitive effective $G$-action $\alpha: G \times M \to M$ such that for each $p \in M$ the stabilizer subgroup $G_p$ is compact and there is a subgroup $\Gamma \leq G$ (acting freely and properly discontinuously on $X$) so that the manifold $\Gamma \backslash X$ is compact. Two geometries $(X_1, G_1, \alpha_1)$ and $(X_2, G_2, \alpha_2)$ are said to be \emph{equivalent} if there is a diffeomorphism $f: M_1 \to M_2$ and a Lie group isomorphism $\Psi: G_1 \to G_2$ that intertwine the two group actions: $f(\alpha_1(g, x)) = \alpha_2(\Psi(g), f(x))$ for any $x \in X$ and $g \in G_1$. In the event the  geometries $(X, H, \beta)$ and $(X, G, \alpha)$ are such that $H$ is a subgroup of $G$ and $\beta$ is the restriction of $\alpha$, we will say that $(X, H, \beta)$ is a \emph{sub-geometry} of $(X, G, \alpha)$ and write $(X, H, \beta) \leq (X, G, \alpha)$. A geometry is said to be \emph{maximal} if (up to equivalence) it is not a proper sub-geometry of another geometry. It can be shown that every geometry is contained in a maximal geometry (cf. \cite[Prop. 1.1.2]{Filipkiewicz}); however, as was recently observed by Geng \cite[p. 7]{Geng}, the maximal geometry need not be unique. To simplify notation, when the $G$-action on $X$ is understood, we will denote the geometry $(X, G, \alpha)$ by $(X,G)$. 

For a geometry $(X,G)$, the requirement that $G$ act with compact stabilizers ensures the space $\mathscr{R}_G(X)$ consisting of the $G$-invariant Riemannian metrics on $X$ is non-empty. We will say that a geometry $(X,G)$ is \emph{metrically maximal}, if whenever $(X,H) \leq (X,G) \lneq (X, L)$, we have $\mathscr{R}_{L}(X) \subsetneq  \mathscr{R}_G(X) = \mathscr{R}_{H}(X)$. Following Scott \cite[p. 403]{Scott}, a complete locally homogeneous metric on a manifold $M$ is called a \emph{geometric structure} and it is said to be \emph{modeled} on the geometry $(X, G)$ if its universal Riemannian cover is isometric to $X$ equipped with a metric in $\mathscr{R}_G(X)$. A geometric structure on $M$ is said to be \emph{maximal} if it is modeled on a maximal geometry.

\begin{ConstantCurvatureGeometry}
Let $\mathbb{E}^n$ be $n$-dimensional Euclidean space, $\mathbb{S}^n$ be the $n$-dimensional sphere equipped with the round metric of constant curvature $+1$ and $\mathbb{H}^n$ the $n$-dimensional upper half-plane $H^n$ equipped with the hyperbolic metric of constant sectional curvature $-1$. Then, we have the following maximal simply-connected $n$-dimensional geometries:
\begin{itemize}
\item $(X = \mathbb{R}^n, G = \Isom(\mathbb{E}^n)^0)$, where (up to isometry) $\mathscr{R}_G(X)$ consists of the unique flat metric on $\R^n$;
\item $(X = S^n , G = \Isom(\mathbb{S}^n)^0)$, where (up to isometry) $\mathscr{R}_G(X)$ consists of the metrics of constant positive sectional curvature on $S^n$; 
\item $(X = H^n, G = \Isom(\mathbb{H}^n)^0)$, where (up to isometry) $\mathscr{R}_G(X)$ consists of the metrics of constant negative sectional curvature on $H^n$,
\end{itemize}
where for any Riemannian manifold $(M,g)$ we let $\Isom(M,g)^0$ denote the connected component of the identity.

\end{ConstantCurvatureGeometry}

Given a geometry $(X,G)$, $X$ equipped with a choice of metric $g \in \mathscr{R}_G(X)$ is a simply-connected homogeneous space, and any quotient of $X$ by a subgroup $\Gamma \leq \Isom(X,g)$ of isometries that acts freely and properly discontinuous admits a locally homogeneous metric. Conversely, a result of Singer states that the universal Riemannian cover of a locally homogeneous manifold is itself a homogeneous space \cite{Singer}. Therefore, to classify the compact $n$-manifolds admitting geometric structures, one should begin by classifying the maximal geometries. And, to understand all the possible locally homogeneous metrics supported by such spaces, one must classify the metrically maximal $n$-dimensional geometries. 
 
In dimension two, the metrically maximal geometries are precisely the two-dimensional geometries of constant sectional curvature described above, these geometries are also maximal. Additionally, the uniformization theorem states that any surface $\Sigma$ admits geometric structures and the locally homogeneous metrics supported by $\Sigma$ are all modeled on the same maximal two-dimensional geometry. Turning to the Laplace spectrum, we recall that Berger has shown that a surface of constant sectional curvature is determined up to local isometry by its spectrum. Said differently, a locally homogeneous surface is determined up to local isometry by its spectrum. The numerous examples of isospectral Riemann surfaces show that this result is optimal. This discussion raises the following questions.

\begin{enumerate}
\item Is local homogeneity an audible property? More specifically, is each locally homogeneous space determined up to local isometry by its spectrum?

\item Let $(M_1, g_1)$ and $(M_2, g_2)$ be two isospectral locally homogeneous $n$-manifolds.
\begin{enumerate}
\item Does it follow that $(M_1, g_1)$ and $(M_2, g_2)$ are modeled on a common geometry?
\item Does it follow that $(M_1, g_1)$ and $(M_2, g_2)$ are locally isometric?
\end{enumerate}
 What if one restricts their attention to maximal geometries?
\end{enumerate}

In general, the answer to these questions is no. Indeed, Szabo has demonstrated that local homogeneity is inaudible in dimension 10 and higher \cite[Section 3]{Szabo}.  Restricting our attention to locally homogeneous spaces, the examples of Gordon  \cite{Gordon93, Gordon94} demonstrate that in dimension 8 and higher isospectral locally homogeneous spaces need not be locally isometric (cf. \cite{GordonWilson99, Schueth01, Sutton, Proctor}). Furthermore, making use of the third author's generalization of Sunada's method \cite{Sutton}, An, Yu and Yu \cite{AYY} demonstrate that isospectral locally homogeneous spaces of dimension at least 26 need not share the same model geometry (cf. \cite{Sutton}). Indeed, they produce examples of isospectral simply-connected homogeneous spaces of dimension at least 26 that are not homeomorphic. We do not know whether any of the examples discussed in this paragraph involve spaces modeled on maximal geometries.

Perelman's resolution of the geometrization conjecture confirms that an orientable prime closed three-manifold admits a canonical decomposition into pieces that each support geometric structures modeled on precisely one of the eight maximal three-dimensional geometries (see below). Therefore, partially inspired by Berger's result regarding the audibility of the local geometry of surfaces of constant sectional curvature and the importance of three-dimensional geometric structures in classifying three-manifolds, we pose the following problem.  

\begin{prob}
To what extent are three-dimensional geometric structures encoded in the spectrum?
Is the property of being modeled on a \emph{maximal} three-dimensional geometry ``audible''? 
Is it the case that isospectral compact locally homogeneous three-manifolds are necessarily modeled on a common geometry $(X,G)$? Are isospectral compact locally homogeneous three-manifolds necessarily locally isometric? Finally, is every compact locally homogeneous three-manifold determined up to local isometry by its spectrum?
\end{prob}

The three-dimensional maximal geometries have been classified by Thurston as follows.

\begin{Thurston}[see \cite{Thurston, Scott}]\label{thm:Thurston} 
A maximal simply-connected three-dimensional geometry $(X, G)$ is equivalent to one of the following eight geometries:
\begin{enumerate}
\item[(T1)] $(\mathbb{R}^3, \Isom(\mathbb{E}^3)^0)$,
\item[(T2)] $(S^3, \Isom(\mathbb{S}^3)^0)$,
\item[(T3)] $(\mathbb{H}^3, \Isom(\mathbb{H}^3)^0)$,
\item[(T4)] $(S^2 \times \mathbb{R}, \Isom(\mathbb{S}^2 \times \mathbb{E})^0)$,
\item[(T5)] $(\mathbb{H}^2 \times \mathbb{R}, \Isom(\mathbb{H}^2 \times \mathbb{E})^0)$,
\item[(T6)] $(\operatorname{Nil}, \Isom(\operatorname{Nil}, g_{\rm{max}})^0)$, where $\Isom(\operatorname{Nil}, g_{\rm{max}})^0$ is four-dimensional, has index two in the full isometry group and is generated by $\operatorname{Nil}$ acting on itself by left translations and an $S^1$-action;
\item[(T7)] $(\widetilde{\operatorname{SL}_2(\mathbb{R})}, \Isom(\widetilde{\operatorname{SL}_2(\mathbb{R})}, g_{\rm{max}})^0)$, where $\Isom(\widetilde{\operatorname{SL}_2(\mathbb{R})}, g_{\rm{max}})^0$ is four-dimensional, has index two in the full isometry group and is generated by $\widetilde{\operatorname{SL}_2(\mathbb{R})}$ acting on itself by left translations and an action by $\mathbb{R}$;
\item[(T8)] $(\operatorname{Sol}, \Isom(\operatorname{Sol}, g_{\rm{max}})^0)$, where $\Isom(\operatorname{Sol}, g_{\rm{max}})^0$ is three-dimensional, has index eight in the full isometry group, and is generated by $\operatorname{Sol}$ acting on itself by left translation.
\end{enumerate} 
For the geometries $(T6)$, $(T7)$ and $(T8)$, the metric $g_{\rm{max}}$ is drawn from the collection of left-invariant metrics on $G$ with the property that the associated isometry group is maximal among the isometry groups of left-invariant metrics on $G$.
\end{Thurston}

\noindent
Additionally, for any compact three-manifold $M$ admitting geometric structures, the locally homogeneous metrics supported by $M$ are all modeled on a sub-geometry of a unique maximal three-dimensional geometry.

It follows from the work of Sekigawa \cite[Theorem B]{Sekigawa} that any proper sub-geometry of a Thurston geometry must be of the form $(G, \Isom(G, g)^0)$, where $g$ is a left-invariant metric on $G$. The work of Raymond and Vasquez shows that the three-dimensional Lie groups giving rise to geometries of this type are precisely $\R^3$, $S^3$, $\operatorname{Nil}$, $\operatorname{Sol}$, $\widetilde{\operatorname{SL}_2(\R)}$ and $\R^2 \rtimes_R \R$, where $R: \R \to \operatorname{Aut}(\R^2)$ is the homomorphism that sends $\theta$ to counterclockwise rotation through $2\pi \theta$ \cite{RaymondVasquez}. The geometries arising from the first five of these groups are each a sub-geometry of an obvious (and unique) Thurston geometry, and the geometry $(\R^3, \R^3)$ clearly gives rise to the Euclidean metric on $\R^3$. 

For the group $\R^2 \rtimes_R \R$, we let $\Psi: \R^2 \rtimes_{R} \R \to \Isom(\mathbb{E}^3) = \R^3 \rtimes O(3)$ be the Lie group embedding defined by $\Psi(v; \theta) = ((v;\theta), R(\theta) \oplus 1)$. One can check that $\Psi(\R^2 \rtimes_R \R)$ acts transitively on $\R^3$ and the geometry $(\R^2 \rtimes_R \R, \R^2 \rtimes_R \R)$ is equivalent to $(\R^3, \Psi(R^2 \rtimes_R \R))$. Therefore, $(\R^2 \rtimes_R \R, \R^2 \rtimes_R \R)$ is a sub-geometry of $(\R^3, \Isom(\mathbb{E}^3)^0)$ that gives rise to flat metrics and metrics of negative Ricci curvature (see \cite[Thm. 1.5 \& Cor. 4.8]{MilnorLieGroups}). In total, there are ten closed three-manifolds---sometimes referred to as ``platycosms'' \cite{DoyleRossetti}---that admit flat metrics and the five orientable manifolds admitting flat metrics with finite cyclic holonomy can be realized in the form $\Gamma \backslash (\R^2 \rtimes_R \R)$ for some $\Gamma \leq \R^2 \rtimes_R \R$ \cite[Table 1]{RaymondVasquez}. In particular, the three-torus admits locally homogeneous metrics modeled on this geometry, some of which are not flat.

From the preceding discussion we deduce the following classification of metrically maximal three-dimensional geometries, which shows the universal Riemannian coverings of closed locally homogeneous three-manifolds come in eight families.

\begin{MetricallyMaximal}\label{thm:SimplyConnectedLocallyHomogeneous}
A metrically maximal three-dimensional geometry is equivalent to one of the following geometries:
\begin{enumerate}
\item[(MM1)] the $\R^2 \rtimes_{R} \R$-geometry $(\R^2 \rtimes_{R} \R, \R^2 \rtimes_{R} \R)$; 
\item[(MM2)] the $S^3$-geometry $(S^3, S^3)$; 
\item[(MM3)] the $\mathbb{H}^3$-geometry $(\mathbb{H}^3, \Isom(\mathbb{H}^3)^0)$;
\item[(MM4)] the $S^2\times R$-geometry $(S^2 \times \mathbb{R}, \Isom(\mathbb{S} \times \mathbb{E})^0)$;
\item[(MM5)] the $\mathbb{H}^2 \times \R$-geometry $(\mathbb{H}^2 \times \mathbb{R}, \Isom(\mathbb{H}^2 \times \mathbb{E})^0)$;
\item[(MM6)] the $\operatorname{Nil}$-geometry $(\operatorname{Nil}, \operatorname{Nil})$;
\item[(MM7)] the $\widetilde{\operatorname{SL}_2(\R)}$-geometry $(\widetilde{\operatorname{SL}_2(\R)}, \widetilde{\operatorname{SL}_2(\R)})$;
\item[(MM8)] the $\operatorname{Sol}$-geometry $(\operatorname{Sol}, \operatorname{Sol})$.
\end{enumerate}
\end{MetricallyMaximal}

\noindent 
Therefore, in contrast with the two-dimensional case, the Thurston geometries $(T3)$, $(T4)$ and $(T5)$ are the only maximal three-dimensional geometries that are also metrically maximal.

\subsection{On the audibility of three-dimensional geometric structures} As we will recall in Section~\ref{sec:Heat}, the heat invariants associated to a closed Riemannian manifold $(M,g)$ form a sequence $\left\{ a_j(M,g) \right\}_{j=0}^{\infty}$ of spectral invariants. That is, two isospectral manifolds must have identical heat invariants. Using these spectral invariants we demonstrate that among compact locally homogeneous three-manifolds (1) a space modeled on the $S^2\times \R$-geometry, $\mathbb{H}^2 \times \R$-geometry, $\operatorname{Nil}$-geometry or $\R^2 \rtimes_{R} \R$-geometry is determined up to local isometry by its spectrum, (2) the property of being modeled on the $S^3$-geometry is audible and any Riemannian manifold modeled on the $S^3$-geometry that is sufficiently close to a metric of constant positive curvature is determined up to local isometry by its spectrum, (3) there is partial evidence that the property of being modeled on the metrically maximal geometry $(\widetilde{\operatorname{SL}_2(\R)}, \widetilde{\operatorname{SL}_2(\R)})$ is audible, and (4) local geometry is audible among spaces modeled on the metrically maximal geometry $(\operatorname{Sol}, \operatorname{Sol})$.

\begin{thm}\label{thm:Main}
For $j =1,2$, let $(M_j, g_j)$ be a compact locally homogeneous three-manifold with Ricci tensor $\operatorname{Ric}_j$ and corresponding vector of Ricci eigenvalues $\nu(g_j) = (\nu_1(g_j), \nu_2(g_j), \nu_3(g_j))$. Now, suppose the first four heat invariants of $(M_1, g_1)$ and $(M_2, g_2)$ agree; i.e., $a_j(M_1, g_1) = a_j(M_2, g_2)$, for $j = 0,1,2,3$.
\begin{enumerate}
\item Suppose $(M_1, g_1)$ is modeled on the $S^2\times \R$-geometry, $\mathbb{H}^2 \times \R$-geometry, $\operatorname{Nil}$-geometry or $\R^2 \rtimes_{R} \R$-geometry. Then, $(M_1, g_1)$ and $(M_2, g_2)$ are locally isometric.

\item Suppose $(M_1, g_1)$ is modeled on the $S^3$-geometry. Then, $(M_2, g_2)$ is also modeled on the $S^3$-geometry and $\Ric$ and $\Ric'$ have the same signature. Furthermore, within the space of left-invariant metrics on $S^3$, there is a neighborhood $\mathcal{U}$ of the round metric such that if the universal Riemannian cover of $(M_1,g_1)$ is isometric to a space in $\mathcal{U}$, then $(M_1, g_1)$ and $(M_2, g_2)$ are locally isometric. 

\item Suppose $(M_1, g_1)$ is modeled on the $\widetilde{\operatorname{SL}_2(\R)}$-geometry. Furthermore, assume $\Ric_1$  has signature $(+,-,-)$ with $P_2(\nu(g_1)) >0$, where $P_2$ is the second elementary symmetric polynomial in three variables.
Then, $(M_2, g_2)$ is also modeled on the $\widetilde{\operatorname{SL}_2(\R)}$-geometry and $\Ric_2$ also has signature $(+,-,-)$ with $P_2(\nu(g_2)) >0$.
If, in addition, the quantity $P_1(\nu(g_1))^2 - 4P_2(\nu(g_1))$ is negative, where $P_1$ and $P_2$ are the first and second elementary symmetric polynomials in three variables, then $(M_1,g_1)$ and $(M_2, g_2)$ are locally isometric.

\item Suppose $(M_1, g_1)$ and $(M_2, g_2)$ are both modeled on the $\operatorname{Sol}$-geometry, then $(M_1, g_1)$ and $(M_2, g_2)$ are locally isometric.
\end{enumerate}
\end{thm} 

Combining the first statement of the preceding theorem with Berger's observation that, among closed Riemannian three-manifolds, a closed three-manifold of constant sectional curvature is determined up to local isometry by its first \emph{three} heat invariants \cite[Theorem 7.1]{Berger}, we discover that the property of being locally symmetric is audible among compact locally homogeneous three-manifolds. In fact, each locally symmetric three-manifold is determined up to local isometry by its spectrum among all compact locally homogeneous three-manifolds. We also note that one can check that up to scaling the space of Riemannian metrics associated to the three-dimensional \emph{maximal} geometry $(T7)$ forms a one-parameter family $\{g_t\}_{t >0}$ of left-invariant metrics on $\widetilde{\operatorname{SL}_2(\R)}$, where the Ricci eigenvalues of $g_t$ are $\nu_1(g_t)=2$ and $\nu_2(g_t) = \nu_3(g_t) =-2(t+1)$. It then follows from Theorem~\ref{thm:Main}(3) that, among locally homogeneous three-manifolds, a space sharing the same first four heat invariants as a space modeled on the maximal geometry $(T7)$ must be modeled on the metrically maximal geometry $(MM7)$. This discussion and the first two statements of Theorem~\ref{thm:Main} can be summarized as follows.

\begin{cor}\label{cor:MainLocSymm}
Among compact locally homogeneous three-manifolds, a compact three-manifold with universal Riemannian cover isometric to  
\begin{enumerate}
\item a symmetric space, 
\item $\R^2 \rtimes \R$ equipped with a left-invariant metric, 
\item $\operatorname{Nil}$ equipped with a left-invariant metric, or
\item $S^3$ equipped with a left-invariant metric sufficiently close to a round metric 
\end{enumerate}
is determined up to local isometry by its first four heat invariants $a_0$, $a_1$, $a_2$ and $a_3$.
Additionally, among compact locally homogeneous three-manifolds, a space sharing the same first four heat invariants as one modeled on the maximal geometry $(\widetilde{\operatorname{SL}_2(\mathbb{R})}, \Isom(\widetilde{\operatorname{SL}_2(\mathbb{R})}, g_{\rm{max}})^0)$ must be modeled on the metrically maximal geometry $(\widetilde{\operatorname{SL}_2(\mathbb{R})}, \widetilde{\operatorname{SL}_2(\mathbb{R})})$.  
\end{cor}

There are four closed three-manifolds admitting geometric structures modeled on $(S^2 \times \R, \Isom(\mathbb{S}^2 \times \mathbb{E})^0)$: namely, $S^2 \times S^1$, $\R P^2 \times S^1$, $\mathbb{R}P^3 \# \R P^3$ and the non-trivial $S^1$-bundle over $\R P^2$. And, each of these spaces possesses a two-dimensional family of locally symmetric metrics. In Section~\ref{sec:S2TimesRGeometries}, we use the fundamental tone (i.e., the first non-zero eigenvalue of the Laplace operator) to show that the compact locally symmetric spaces modeled on $(S^2 \times \R, \Isom(\mathbb{S}^2 \times \mathbb{E})^0)$ can be mutually distinguished by their spectra (see Proposition~\ref{prop:S2TimesRSpectra}). Combining this observation with the preceding corollary, we find that among compact locally homogeneous spaces, any closed Riemannian manifold modeled on $(S^2 \times \R, \Isom(\mathbb{S}^2 \times \mathbb{E})^0)$ is uniquely determined by its spectrum. 

\begin{cor}\label{cor:MainS2TimesR}
Among compact locally homogeneous three-manifolds, a compact locally symmetric three-manifold modeled on the metrically maximal geometry $(S^2 \times \R, \Isom(\mathbb{S}^2 \times \mathbb{E})^0)$ is determined up to isometry by its spectrum.
\end{cor}

\noindent
This implies that, up to scaling, the common Riemannian covering of a non-trivial isospectral pair of compact locally symmetric three-manifolds must be $\mathbb{E}^3$, $\mathbb{H}^{3}$ or $\mathbb{H}^{2} \times \mathbb{E}$. We note that while the flat three-dimensional tori can be mutually distinguished by their spectra \cite{Schiemann}, $\mathbb{E}^3$ covers a unique isospectral pair known as ``tetra and didi'' \cite{DoyleRossetti, RossettiConway}. In contrast, the literature contains numerous examples of isospectral pairs covered by $\mathbb{H}^{3}$ or $\mathbb{H}^{2} \times \mathbb{E}$ \cite{Sunada, Vigneras, Buser, Reid}. 

Focusing on nilmanifolds, we recall that by explicitly computing the spectra of three-dimensional Riemannian nilmanifolds, Gordon and Wilson have shown that compact three-manifolds modeled on the $\operatorname{Nil}$-geometry can be mutually distinguished via their spectra \cite{GordonWilson86}. Applying Theorem~\ref{thm:Main}, it follows that three-dimensional Riemannian nilmanifolds are determined up to isometry by their spectra among compact locally homogeneous three-manifolds.

\begin{cor}\label{cor:MainNilGeometry}
Among compact locally homogeneous three-manifolds, a compact locally homogeneous three-manifold modeled on the metrically maximal geometry $(\operatorname{Nil}, \operatorname{Nil})$ is determined up to isometry by its spectrum. 
\end{cor}

\subsection{Can you hear the local geometry of a three-manifold?} The results presented in this article and its sequel \cite{LSS2} provide strong evidence that one should expect a compact locally homogeneous three-manifold to be determined up to local isometry by its spectrum. The broader question undergirding this article is whether the local geometry of any \emph{low-dimensional} closed Riemannian manifold is encoded in its spectrum. 

Indeed, the earliest examples of isospectral Riemannian manifolds all have a common Riemannian covering. For example, isospectral pairs arising from the original incarnation of Sunada's method necessarily have a common Riemannian covering \cite{Sunada}. This coupled with the fact that the heat invariants are averages of local geometric data makes it seem plausible that the universal Riemannian cover of a closed Riemannian manifold is audible. However, as we noted previously, in 1993, Carolyn Gordon produced the first examples of closed isospectral manifolds that are not locally isometric \cite{Gordon93, Gordon94} via a construction inspired by Szabo's approach to building isospectral, yet locally non-isometric manifolds with boundary \cite{Szabo}.\footnote{Szabo discovered his examples prior to Gordon's 1993 paper, but his result remained unpublished until 1999.} The ensuing years have seen many more examples of isospectral, yet locally non-isometric closed manifolds \cite{Schueth99, Proctor, Schueth01S5}, including surprising pairs arising from the third named author's generalization of Sunada's method \cite{Sutton, AYY}. And, in 2001, Schueth's examples of isospectral metrics on $S^2 \times T^2$ \cite{Schueth01} demonstrated that the local geometry of a closed Riemannian manifold of dimension at least four need not be encoded in its spectrum, leaving open the following problem.

\begin{prob}
Is the local geometry of a closed Riemannian manifold of dimension two or three ``audible''?
\end{prob}

\begin{Structure}
In Section \ref{sec:Heat}, we first review the formulae for the first four heat invariants of a Riemannian manifold. Then, with a few exceptions, we find the first four heat invariants of a compact locally homogeneous three-manifold can be expressed as the product of its volume with a symmetric rational function of the eigenvalues of the accompanying Ricci tensor. The main goal of the section is to establish Theorem~\ref{thm:P3}, which severely restricts the possible Ricci-eigenvalues possessed by isospectral three-manifolds with Riemannian universal cover isometric to a unimodular Lie group equipped with a left-invariant metric. Theorem~\ref{thm:P3} is then used alongside Proposition~\ref{prop:ObservationE}, in Section~\ref{sec:AudibleModelGeometries}, to prove Theorem~\ref{thm:Main}, which describes audibility results concerning various three-dimensional geometries. Finally, in Section~\ref{sec:S2TimesRGeometries}, we let $\mathbb{S}_{k}^{n}$ denote the round $n$-sphere of constant sectional curvature $k>0$ and compute the Laplace spectra of manifolds having universal Riemannian cover $\mathbb{S}_{k}^{2} \times \mathbb{E}$. Then, by applying Theorem~\ref{thm:Main} and comparing fundamental tones, we show that Riemannian manifolds modeled on the $S^2 \times \R$-geometry are determined up to isometry by their Laplace spectra among compact locally homogeneous three-manifolds.

\end{Structure}

\begin{Acknowledgments}
The authors would like to thank Dorothee Schueth for helpful suggestions regarding a previous draft of this article; especially, bringing our attention to Lemma \ref{lem:ObservationD} and Proposition \ref{prop:ObservationE}.
\end{Acknowledgments}


\section{\bf Heat Invariants and Locally Homogeneous Three-Manifolds}\label{sec:Heat}

In this section we will review the \emph{heat invariants}---the main analytical tool in this article--- and derive computational simplifications that occur when considering locally homogeneous three-manifolds. Of particular interest to us will be the fact that the heat invariants of a locally homogeneous three-manifold covered by a unimodular Lie group are \emph{almost} symmetric functions in the sectional curvatures $K_{12}$, $K_{13}$ and $K_{23}$ determined by a choice of \emph{Milnor frame} (see Definition~\ref{dfn:MilnorFrame}). This will allow us to express the heat invariants of a locally homogeneous three-manifold as a symmetric function of the Ricci-eigenvalues.

\subsection{The heat invariants}
The Laplace-Beltrami operator of a closed and connected Riemannian $n$-manifold $(M,g)$ is the (essentially) self-adjoint operator $\Delta_g \equiv - \operatorname{div} \circ \operatorname{grad}_g$ on $L^{2}(M, \nu_g)$.  The sequence 
$\lambda_0 = 0 < \lambda_1 \leq \lambda_2 \leq \cdots \nearrow \infty$ of eigenvalues of $\Delta_g$, repeated according to multiplicity, is the \emph{spectrum} of $(M,g)$ and 
we will say that two manifolds are \emph{isospectral} when their spectra agree. Letting $\{ \phi_k \}$ be an orthonormal basis of $L^{2}(M,\nu_g)$ consisting of $\Delta_g$-eigenfunctions, then for each $t >0$ we may define 
$e^{-t \Delta_g} : L^{2}(M, \nu_g) \to L^2(M, \nu_g)$ to be the linear extension of $e^{-t\Delta_g}\phi_k = \lambda_k \phi_k$. Then, $\{e^{-t\Delta_g} \}_{t > 0}$ is a family of self-adjoint operators known as the heat semi-group. 

The operators forming the heat semi-group are trace class (cf. \cite[Thm. V.3]{Berard}) and we have the following asymptotic expansion for the trace of the heat semi-group \cite{Mina}: 
$$\operatorname{Tr}(e^{-t\Delta_g}) = \sum_{k=0}^{\infty} e^{-t \lambda_j}\stackrel{t \searrow 0}{\sim}(4 \pi t)^{-n/2} \sum_{m = 0}^{\infty} a_m(M,g) t^m.$$
The coefficients $\{a_m(M,g) \}_{m=0}^{\infty}$ in this expression are the \emph{heat invariants} of $(M,g)$ and they are spectral invariants; i.e., isospectral manifolds have equal heat invariants.  There are universal polynomials in the components of the curvature tensor and its covariant derivatives, $u_m(M,g)$, such that $a_m(M,g) = \int_{M} u_m(M,g) \, d\nu_g$  \cite[p. 145]{Berard} or \cite[Chp. VI.5]{Sakai2}.  Explicit formulae for the heat invariants are known in few cases (cf. \cite{Polterovich}). 

Let $\nabla$, $R=(R^{i}_{jkl})$, $\Ric=(\rho_{jl}=R_{jil}^{i})$, $\Scal=(g^{jl} \rho_{jl})$, and $\nu_g$ denote the Levi-Civita connection, Riemannian curvature tensor, Ricci curvature tensor, scalar curvature, and Riemannian density, respectively. We follow the sign convention 
for the curvature tensor in \cite{Tanno73} and \cite{Sakai}; namely, for smooth vector fields $X,Y,Z \in \chi(M)$ 
\begin{equation}\label{R}
R(X,Y)Z=\nabla_{[X,Y]}Z-[\nabla_X,\nabla_Y]Z.
\end{equation} 
Consequently, the sectional curvature of the plane spanned by two orthogonal unit vectors $X, Y \in T_pM$ is given by $R(X,Y,X,Y)$.
The first four heat invariants are given by (\cite{Tanno80}): 
\begin{equation}\label{heat0}
a_0(M, g)=\vol(M,g)=\int_{M} 1\, d\nu_g,
\end{equation}

\begin{equation}\label{heat1}
a_1(M, g)=\frac{1}{6} \int_M \Scal \, d\nu_g,
\end{equation}

\begin{equation}\label{heat2}
a_2(M, g)=\frac{1}{360} \int_M 2(\vert R \vert^2-\vert \Ric \vert^2)+5\Scal^2\, d\nu_g,
\end{equation} 

\noindent
and

\begin{equation}\label{heat3}
a_3(M, g)=\frac{1}{6!}\int_M \left( \overline{D} + \bar{A} +\frac{2}{3} \Scal (\vert R\vert^2-\vert \Ric \vert^2)+\frac{5}{9}\Scal^3 \right) \, d\nu_g,
\end{equation}

\noindent 
where $\overline{D}$ is defined by 
\begin{equation}\label{eqn:Dbar}
\overline{D} = - \frac{1}{9}\vert \nabla R \vert^2-\frac{26}{63}\vert \nabla \Ric \vert^2-\frac{142}{63}\vert \nabla \Scal \vert^2,
\end{equation}

\noindent 
and $\bar{A}$ is defined by 
\begin{equation}\label{eqn:Abar}
\bar{A}=\frac{8}{21}(R,R,R)-\frac{8}{63}(\Ric;R,R)+\frac{20}{63}(\Ric;\Ric;R)-\frac{4}{7}(\Ric \Ric \Ric),
\end{equation}

\noindent 
where, for tensor fields $P=(P_{ijkl})$, $Q=(Q_{ijkl})$, $T=(T_{ijkl})$, $U=(U_{ij})$, $V=(V_{ij})$, and $W=(W_{ij})$ on $(M,g)$, we have the following products 

$$(P,Q)=P_{ijkl}Q^{ijkl},$$
$$\vert P \vert^2=(P,P),$$
$$(P,Q,T)=P^{ij}_{kl}Q^{kl}_{rs}T^{rs}_{ij},$$
$$(U;Q,T)=U^{rs}Q_{rjkl}T_{s}^{jkl},$$
$$(U;V;T)=U^{ab}V^{cd}T_{abcd},$$
$$(UVW)=U^{i}_{j}V^{j}_{k}W^{k}_{i}.$$

\begin{rem}\label{rem:Orthogonality}
For $i =1,2$, let $(M_j, g_j)$ be a Riemannian manifold with tensor fields $P_j$, $Q_j$, $T_j$, $U_j$, $V_j$ and $W_j$ as above, and let $P = P_1+P_2, T= T_1+T_2, Q =Q_1 + Q_2, U= U_1 + U_2, V= V_1 + V_2$ and $W= W_1 + W_2$ be their orthogonal sums on the product manifold $(M_1, \times M_2, g_1 \times g_2)$. Then, 
$(P, Q) = \sum (P_j, Q_j)$, $(P,Q,T) = \sum (P_j, Q_j, T_j)$, $(U;Q,T) = \sum (U_j; Q_j, T_j)$, 
$(U; V; T) = \sum (U_j; V_j; T_j)$ and $(UVW) = \sum (U_jV_jW_j)$
\end{rem}

\begin{rem}\label{rem:Dbar}
When $(M,g)$ is locally homogeneous, $\Scal$ is constant, which implies $\overline{D} = - \frac{1}{9}\vert \nabla R \vert^2-\frac{26}{63}\vert \nabla \Ric \vert^2$. Furthermore, when $(M,g)$ is locally symmetric
$\overline{D}$ is identically zero, since $\nabla R$ and  $\nabla \Ric$  both vanish.
\end{rem}

The heat invariants have been used to prove many interesting spectral rigidity results. For instance, we have the following theorem demonstrating that constant curvature is an audible property in low dimensions.

\begin{thm}[Berger \cite{Berger}, Tanno \cite{Tanno73}]\label{thm:BergerTanno}
Let $(M,g)$ and $(M',g')$ be compact manifolds of dimension $2 \leq n \leq 5$ such that $a_j(M,g) = a_j(M',g')$ for $j =0,1,2,3$. And, fix a real number $K$. Then, $(M,g)$ is a space of constant sectional curvature $K$ if and only if $(M',g')$ is a space of constant sectional curvature $K$.
\end{thm}
 
 \noindent
In the case where the dimension is two or three, this theorem was observed to be true by Berger under the milder assumption that only the first three heat invariants agree \cite[Theorem 7.1]{Berger}. 

\subsection{The geometry of locally homogeneous three-manifolds.}

We begin by recalling the following well-known fact.

\begin{lem}\label{lem:RicciDiagonalization}
Let $(M,g)$ be a Riemannian manifold and for $p \in M$ let $\{ e_1, \ldots , e_n \}$ be an orthonormal basis of $T_pM$. If $ \langle R(e_i, e_j)e_k, e_l \rangle = 0$ whenever three of the indices are pairwise distinct, then $\{e_1, \ldots , e_n\}$ diagonalizes the Ricci tensor $\Ric(\cdot)$. And, the converse is true when $M$ is three-dimensional.
\end{lem}

Recall that the simply-connected and connected unimodular three-dimensional Lie groups are $\R^3$, $S^3$, $\widetilde{\operatorname{SL}_2(R)}$, $\operatorname{Nil}$, $\operatorname{Sol}$ and $\R^2 \rtimes_R \R = \widetilde{\Isom (\mathbb{E}^2)^0}$, where $R: \theta \in \R \mapsto R(2\pi\theta) \in \SO(2)$ is counterclockwise rotation through $2\pi\theta$. The next result can be deduced easily from \cite{MilnorLieGroups}.

\begin{lem}\label{lem:MilnorFrame}
Let $G$ be one of the six simply-connected three-dimensional unimodular Lie groups. Given a three-manifold  $(M,g)$ locally isometric to $G$ equipped with a left-invariant metric, any orthonormal basis $\{e_1,e_2, e_3\}$ of $T_pM$ consisting of Ric-eigenvectors extends to a local framing $\{E_1, E_2, E_3\}$ on a neighborhood $U$ of $p$ such that 

\begin{enumerate}
\item $\{E_1, E_2, E_3\}$ is orthonormal; 
\item there are constants $\lambda_1$, $\lambda_2$ and $\lambda_3$ such that 
$$[E_1, E_2] = \lambda_3 E_3, \;  [E_2, E_3] = \lambda_1 E_1, \mbox{ and }  [E_3, E_1] = \lambda_2 E_2;$$
\item $\{E_1, E_2, E_3\}$ diagonalizes the Ricci tensor:
$$\Ric(E_1) \equiv \nu_1 = 2 \mu_2 \mu_3, \; \Ric(E_2) \equiv \nu_2 = 2 \mu_1 \mu_3, \mbox{ and } \Ric(E_3) \equiv \nu_3 = 2 \mu_1 \mu_2,$$
where $\mu_i \equiv \frac{1}{2}(\lambda_1 + \lambda_2 + \lambda_3) - \lambda_i$ for $i = 1,2, 3$.
\item $R(E_i, E_j, E_k, E_l)$ and $\operatorname{Ric}(E_i, E_j)$ are constant for all choices of $i, j, k$ and $l$.
\end{enumerate}
\end{lem}

\begin{dfn}\label{dfn:MilnorFrame}
Let $(M,g)$ be a locally homogeneous three-manifold locally isometric to a unimodular Lie group $G$ equipped with a left-invariant metric. And, let $\{e_1, e_2, e_3\}$ be an orthonormal basis of $T_pM$ consisting of eigenvectors of the Ricci tensor, for some $p \in M$. An extension $\{E_1, E_2, E_3\}$ of  $\{e_1, e_2, e_3\}$ to a neighborhood $U$ of $p$ as in Lemma~\ref{lem:MilnorFrame} will be called a \emph{Milnor frame}. 
\end{dfn}

\begin{cor}\label{cor:MilnorFrame}
Let $(M,g)$ be a locally homogeneous three-manifold locally isometric to a unimodular Lie group $G$ equipped with a left-invariant metric and let $\{E_1, E_2, E_3\}$ be a Milnor frame in a neighborhood of some $p \in M$. Then,

\begin{enumerate}
\item $\nabla_{E_j}E_j = 0$ for $j = 1, 2, 3$
\item $\mu_1 = \Gamma_{12}^{3} = -\Gamma_{13}^{2}$, $\mu_2 = \Gamma_{23}^{1} = -\Gamma_{21}^{3}$ and 
$\mu_3 = \Gamma_{31}^{2} = -\Gamma_{32}^{1}$
\item $\nabla_{E_{\sigma(1)}}E_{\sigma(2)} =  \Gamma_{\sigma(1)\sigma(2)}^{\sigma(3)} E_{\sigma(3)}$, where $\sigma$ is a cyclic permutation.
\end{enumerate}
\end{cor}

\begin{rem}
Lemma~\ref{lem:MilnorFrame} and Corollary~\ref{cor:MilnorFrame} will be useful in Proposition~\ref{prop:NormNablaRNablaRic} where we compute expressions for $|\nabla R|^2$ and $|\nabla \operatorname{Ric} |^2$ for a locally homogeneous three-manifold modeled on a unimodular Lie group equipped with a left-invariant metric. 
\end{rem}

The constants $\lambda_1$, $\lambda_2$ and $\lambda_3$ in Lemma~\ref{lem:MilnorFrame} are known as the \emph{structure constants}. Milnor showed the three-dimensional unimodular Lie groups can be classified according to the sign (plus, minus, or zero) of these structure constants \cite[Sec. 4]{MilnorLieGroups}. Milnor also made the following observation concerning the signature of the Ricci tensor of left-invariant metrics on the non-abelian unimodular Lie groups. 

\begin{lem} \label{lem:Ricci_sig}
\cite[Section 4]{MilnorLieGroups}  
Let $(G, g)$ be a simply-connected three-dimensional non-abelian unimodular Lie group equipped with a left-invariant metric $g$, and let $\Ric$ denote its associated Ricci tensor.  
\begin{enumerate}
\item  If $G = S^3$, then (up to a reordering of the Ricci eigenvalues) the signature of $\Ric$ is $(+, +, +)$, $(+,0,0)$, or $(+,-,-)$, and all such signatures occur.
\item If $G = \mathrm{Nil}$, then (up to a reordering of the Ricci eigenvalues) the signature of $\Ric$  is $(+,-,-)$ and the scalar curvature is strictly negative.
\item If $G$ is $\widetilde{\mathrm{SL}(2, \mathbb{R})}$ or $\mathrm{Sol}$, then (up to a reordering of the Ricci eigenvalues) the signatures of $\Ric$ is 
$(+,-,-)$ or $(0,0,-)$; however, the scalar curvature is always negative.
\item If $G$ is $\R^2 \rtimes_R \R$, then $G$ admits a flat left-invariant metric and (up to reordering of the Ricci eigenvalues) every non-flat left-invariant metric on $G$ has Ricci signature $(+,-,-)$ and negative scalar curvature.
 \end{enumerate}
\end{lem}

Since each three-dimensional non-abelian unimodular Lie group supports a left-invariant metric possessing a Ricci tensor of signature $(+, -,-)$, the three-dimensional non-abelian unimodular Lie groups cannot be distinguished through the signatures of the Ricci tensors of their respective left-invariant metrics. However, the following lemma allows us to deduce that from the eigenvalues of the Ricci tensor one may recover the model geometry of a three-dimensional manifold locally isometric to a non-abelian Lie group equipped with a left-invariant metric. 
 
\begin{lem} \label {lem:su_absolute}
Let $(M,g)$ be a compact locally homogeneous three-manifold modeled on the geometry $(G,G)$, where $G$ is a  three-dimensional non-abelian simply-connected unimodular Lie group. Suppose further that the signature of the Ricci tensor of $(M,g)$ is $(+, -, -)$, where without loss of generality, we assume that $\nu_1>0 >\nu_2\geq\nu_3$. Then, $G$ is 
\begin{enumerate}
\item
$\mathrm{SU}(2)$ if and only if $\nu_1>|\nu_3|,$
\item $\widetilde{\mathrm{SL}(2, \mathbb{R})}$ if and only if $\nu_1<|\nu_2|$ or $|\nu_2|<\nu_1<|\nu_3|$,
\item $\mathrm{Sol}$ if and only if $\nu_1=|\nu_2|<\nu_3$, 
\item $\R^2 \rtimes_R \R$ if and only if $\nu_1=|\nu_3|>|\nu_2|$, and
\item $\mathrm{Nil}$ if and only if $\nu_1=|\nu_2|=|\nu_3|$.
\end{enumerate}
\end{lem}

\begin{rem}
It follows from Lemma~\ref{lem:CrossProduct} that the components of any $\nu= (\nu_1, \nu_2, \nu_3) \in \R^3$ satisfying $\nu_1> 0 > \nu_2 \geq \nu_3$ are the eigenvalues of a left-invariant metric on a unimodular Lie group $G$. 
\end{rem}

\begin{proof}
The signs of $\nu_1, \nu_2, \nu_3$ imply that the signs of $(\mu_1, \mu_2, \mu_3)$ are either $(-,+,+)$ or $(+,-,-)$.

When  the signs of $(\mu_1, \mu_2, \mu_3)$ are $(+, -,-)$ 
\begin{equation}\label{mu1}
\lambda_1=\Big( \frac{\nu_1 \nu_3}{2\nu_2}\Big)^{1/2} + \Big( \frac{\nu_1 \nu_2}{2\nu_3}\Big)^{1/2},
\end{equation}
\begin{equation}\label{mu2}
\lambda_2= -\Big( \frac{\nu_2 \nu_3}{2\nu_1}\Big)^{1/2} + \Big( \frac{\nu_1 \nu_2}{2\nu_3}\Big)^{1/2},
\end{equation}
and
\begin{equation}\label{mu3}
\lambda_3= - \Big( \frac{\nu_2 \nu_3}{2\nu_1}\Big)^{1/2} +\Big( \frac{\nu_1 \nu_3}{2\nu_2}\Big)^{1/2}.
\end{equation}

The equation \eqref{mu1} implies that $\lambda_1 > 0$. The lemma now follows by the classification of unimodular three-dimensional Lie groups in terms of the signs of $\lambda_1, \lambda_2, \lambda_3$ (see 
\cite[p. 307]{MilnorLieGroups}).

The proof for the case when $(\mu_1, \mu_2, \mu_3)=(-, +, +)$ is similar.
\end{proof}

\begin{dfn}\label{dfn:Multiset}
A \emph{multiset} in $\R$ is a map $m: \R \to \N \cup \{0\}$, where we think of $m(x)$ as the multiplicity of $x$ in the multiset. A multiset $m$ is said to be a \emph{$k$-multiset}, for $k \in \N$, if $m$ is non-zero at finitely many distinct values $x_1, \ldots , x_q$ and $\sum m(x_j) = k$. A $k$-multiset $m$ will be denoted by $[x_{11}, \ldots x_{1m(1)}, \ldots , x_{q1} \ldots , x_{qm(q)} ]$, where $x_{ij} = x_i$ for $j = 1, \ldots , m(x_i)$. The collection of $k$-multisets consisting of positive numbers will be denoted by $\mathscr{M}^{+}_k$. 
\end{dfn}

The following proposition shows that among spaces locally isometric to a non-abelian Lie group equipped with a left-invariant metric, the multiset of Ricci eigenvalues determines the model geometry.

\begin{prop}
For $j = 1,2$, let $(M_j,g_j)$ be a locally homogeneous three-manifold modeled on the geometry $(G_j,G_j)$, where $G_j$ is a simply-connected unimodular Lie group, and let $\Ric_j$ be the associated Ricci tensor with eigenvalues $\nu_1(g_j)$, $\nu_2(g_j)$ and $\nu_3(g_j)$.
If $[\nu_1(g_1), \nu_2(g_1), \nu_3(g_1)] = [\nu_1(g_2), \nu_2(g_2), \nu_3(g_2)]$, then $G_1$ and $G_2$ are isomorphic Lie groups. 
\end{prop}

\begin{proof}
Follows directly from Lemmas~\ref{lem:Ricci_sig} and \ref{lem:su_absolute}.
\end{proof}

The previous proposition suggests that a possible strategy for recovering model geometries from spectral data is to recover the Ricci eigenvalues from the heat invariants or other spectral invariants. In the next section we lay the groundwork for this plan. 

\subsection{Heat invariants of locally homogeneous three-manifolds.} In this section we will discover that the heat invariants $a_1$, $a_2$ and $a_3$ of a compact locally homogeneous three-manifold can be expressed as symmetric functions in the eigenvalues of the Ricci tensor with respect to a Milnor frame. These expressions will be key to our arguments. We begin with an observation regarding any three-manifold.

\begin{prop}\label{prop:TensorProducts}
Let $(M,g)$ be a Riemannian three-manifold and let $\{E_1, E_2, E_3\}$ be a local orthonormal framing on a neighborhood $U$ of $p \in  M$ that diagonalizes Ricci. Then, letting 
$K_{ij}(q) \equiv \operatorname{Sec}(E_{iq},E_{jq})$ for $1\leq i < j \leq 3$, we have the following expressions on $U$:
\begin{equation}\label{S}
\Scal=2\{K_{12}+K_{13}+K_{23}\},
\end{equation}  
\begin{equation}\label{normsquaredR}
\vert R \vert^2=4\{(K_{12})^2+(K_{13})^2+(K_{23})^2\},
\end{equation}
\begin{equation}\label{normsquaredrho}
\vert \Ric \vert^2=(K_{12}+K_{13})^2+(K_{12}+K_{23})^2+(K_{13}+K_{23})^2.
\end{equation}
\begin{equation}\label{(R,R,R)}
(R,R,R)=8\{(K_{12})^3+(K_{13})^3+(K_{23})^3\},
\end{equation}
\begin{eqnarray}\label{(rho;R,R)}
(\Ric;R,R) &=& 2\{(K_{12}+K_{13})[(K_{12})^2+(K_{13})^2]\\
& &  + (K_{12}+K_{23})[(K_{12})^2 +(K_{23})^2] \nonumber \\
& & + (K_{13}+K_{23})[(K_{13})^2+(K_{23})^2]\}, \nonumber
\end{eqnarray}
\begin{eqnarray}\label{(rho;rho;R)}
(\Ric;\Ric;R) &=& 2\{K_{12}(K_{12}+K_{13})(K_{12}+K_{23}) \\
& & +K_{13}(K_{12}+K_{13})(K_{13}+K_{23}) \nonumber \\
& & +K_{23}(K_{12}+K_{23})(K_{13}+K_{23})\}, \nonumber
\end{eqnarray}
\begin{equation}\label{(rhorhorho)}
(\Ric \Ric \Ric)=(K_{12}+K_{13})^3+(K_{12}+K_{23})^3+(K_{13}+K_{23})^3.
\end{equation}
\end{prop}

\begin{proof}
A long, yet straightforward computation relying on the chosen frame.
\end{proof}

\begin{rem}
It follows that for a three-manifold $(M,g)$ the integrand of each of the heat invariants $a_0(M,g)$,  $a_1(M,g)$ and $a_2(M,g)$ can be expressed locally as a  symmetric polynomial in the principal curvatures. 
\end{rem}

For the remainder of the paper we will let $P_1(x,y,z)$, $P_2(x,y,z)$ and $P_3(x,y,z)$ be the elementary symmetric polynomials in three variables:
$$P_1(x,y,z) = x+y+z, \; P_2(x,y,z) = xy + xz + yz, \; \textrm{ and } P_3(x,y,z) = xyz.$$

\begin{prop}\label{prop:NormNablaRNablaRic}
Let $(M,g)$ be a locally homogeneous three-manifold locally isometric to a unimodular Lie group $G$ equipped with a left-invariant metric. Let $\{E_1, E_2, E_3 \}$ be a Milnor frame (see Definition~\ref{dfn:MilnorFrame}), $K_{ij} \equiv \operatorname{Sec}(E_i, E_j)$ for $1 \leq i < j \leq 3$ and $\Gamma_{ij}^{k} \equiv \langle \nabla_{E_i}E_j , E_k \rangle$. Then, we have the following expressions:

\begin{eqnarray}
K_{\sigma(1)\sigma(2)} &=& \frac{1}{2} \left( P_1(\nu_1, \nu_2, \nu_3) - 2 \nu_{\sigma(3)} \right),
\end{eqnarray}

\noindent 
where $\sigma$ is any permutation on three elements,

\begin{equation}\label{normsquarednablarho}
-\frac{14}{3} \overline{D} = 4 \vert \nabla \Ric \vert^2=\vert \nabla R \vert^2,
\end{equation}

\noindent
and 

\begin{eqnarray}\label{normsquarednablaR}
\vert \nabla R \vert^2 &=& 8\{(\Gamma_{12}^3K_{13}+\Gamma_{13}^{2}K_{12})^2 + (\Gamma_{21}^3K_{23} +\Gamma_{23}^1K_{12})^2 \\
& & +(\Gamma_{31}^{2}K_{23}+\Gamma_{32}^1 K_{13})^2\}. \nonumber
\end{eqnarray}

\end{prop}

\begin{proof}
A long, yet straightforward computation relying on Lemma~\ref{lem:MilnorFrame} and Corollary~\ref{cor:MilnorFrame}.
\end{proof}

\begin{cor}\label{cor:NormNablaRNablaRic}
Let $(M,g)$ be a locally homogeneous three-manifold.
Fix $p \in M$ and let $\{e_1, e_2, e_3\}$ be an orthonormal basis of $T_pM$ consisting of eigenvectors for the Ricci tensor with associated eigenvalues $\nu= (\nu_1, \nu_2, \nu_3)$. Let $K_{12} = R(e_1, e_2, e_1, e_2)$, $K_{13} = R(e_1, e_3, e_1, e_3)$ and $K_{23} = R(e_2, e_3, e_2, e_3)$ be the associated principal curvatures. 

\begin{enumerate}
\item Then, we have the following relationship between the principal curvatures and the eigenvalues of the Ricci tensor:  

\begin{equation}
 \left(\begin{array}{c}K_{12} \\K_{13} \\K_{23}\end{array}\right) 
 = \frac{1}{2} \left(\begin{array}{ccc}-1 & 1 & 1 \\1 & -1 & 1 \\1 & 1 & -1\end{array}\right) \left(\begin{array}{c} \nu_3 \\ \nu_2 \\ \nu_1 \end{array}\right),
 \end{equation}

and we have the following expressions:

\begin{equation}\label{eqn:SymS}
\Scal=
P_1(\nu), 
\end{equation}  

\begin{equation}\label{eqn:SymNormsquaredR}
\vert R \vert^2=
3P_1^2(\nu) - 8P_2(\nu), 
\end{equation}

\begin{equation}\label{eqn:SymNormsquaredRic}
\vert \Ric \vert^2= 
P_1^2(\nu) - 2P_2(\nu), 
\end{equation}

\begin{equation}\label{eqn:SymR,R,R}
(R,R,R)=
P_1^3(\nu) - 24P_3(\nu), 
\end{equation}

\begin{equation}\label{eqn:SymRic;RR}
(\Ric;R,R) =
-6P_3(\nu) +P_1^3(\nu) -2P_1(\nu)P_2(\nu), 
\end{equation}

\begin{equation}\label{eqn:SymRic;Ric;R}
(\Ric;\Ric;R) = 
P_1(\nu)P_2(\nu) - 6P_3(\nu), 
\end{equation}

\begin{equation}\label{eqn:SymRicRicRic}
(\Ric \Ric \Ric)=
3P_3(\nu) + P_1^3(\nu) -3P_1(\nu)P_2(\nu), 
\end{equation}

\begin{equation}\label{eqn:SymAbar}
\overline{A}=
\frac{16}{63} \left( -\frac{10}{8}P_1^3(\nu) - \frac{189}{4}P_3(\nu) + 9P_1(\nu)P_2(\nu) \right), 
\end{equation}

\begin{equation}\label{eqn:SymAbarNonDeriv}
\overline{A} + \frac{2}{3} \Scal (|R|^2 - |\Ric|^2) + \frac{5}{9} \Scal^3=
 \frac{11}{7}P_1^3(\nu) - 12P_3(\nu) - \frac{12}{7} P_1(\nu)P_2(\nu).
\end{equation}

\item Now, assume that $(M,g)$ is modeled on a unimodular Lie group equipped with a left-invariant metric for which $P_3(\nu) \neq 0$; i.e., all the Ricci eigenvalues are non-zero. Then, we have the following expressions for terms in the integrand of the heat invariants involving the covariant derivatives:

\begin{equation}\label{eqn:SymNormsquarednablaR}
\vert \nabla R \vert^2 = 
-36P_3(\nu) + 40P_1(\nu)P_2(\nu) - 8P_1^3(\nu) + 4\left( \frac{P_1^2(\nu)P_2^2(\nu) - 4P_2^3(\nu)}{P_3(\nu)}
\right)
\end{equation}

\noindent
and

\begin{equation}\label{eqn:SymDbar}
\overline{D}=
\frac{54}{7}P_3(\nu) - \frac{60}{7}P_1(\nu)P_2(\nu) + \frac{12}{7}P_1^3(\nu) - \frac{6}{7}\left( \frac{P_1^2(\nu)P_2^2(\nu) - 4P_2^3(\nu)}{P_3(\nu)} \right).
\end{equation}

\end{enumerate}

\end{cor}

The previous two propositions combine to give us the following expressions for the heat invariants of a locally homogeneous three-manifold as symmetric functions in the eigenvalues of the Ricci tensor. 

\begin{thm}\label{thm:HeatInvariantsSymmetricPolys}
Let $(M,g)$ be a locally homogeneous three-manifold and let  $\nu = (\nu_1, \nu_2, \nu_3)$ be the eigenvalues of the associated Ricci tensor. Then, the heat invariants may be computed in terms of the Ricci eigenvalues as follows.

\begin{equation}\label{eqn:SymA1Ric}
a_1(M,g) = \frac{a_0(M,g)}{6}P_1(\nu)
\end{equation}

\begin{equation}\label{eqn:SymA2Ric}
a_2(M,g) = \frac{a_0(M,g)}{360} \left( 9P_1^2(\nu) -12P_2(\nu) \right)
\end{equation}

\noindent
If, in addition, $(M,g)$ is covered by a unimodular Lie group equipped with a left-invariant metric for which all eigenvalues of the Ricci tensor are non-zero, then 

\begin{equation}\label{eqn:SymA3Ric}
a_3(M,g) = \frac{a_0(M,g)}{7!} \left( 23P_1^3(\nu) + -30P_3(\nu) - 72P_1(\nu)P_2(\nu)  -\frac{6P_1^2(\nu)P_2^2(\nu) - 24P_2^3(\nu)}{P_3(\nu)} \right).
\end{equation}

\end{thm}

\begin{rem}
In the sequel to this article \cite{LSS2}, which concentrates on elliptic three-manifolds, we will find it advantageous to express the heat invariants as symmetric functions of the Christoffel symbols rather than the Ricci eigenvalues.
\end{rem}

\begin{dfn}\label{dfn:ModifiedHeatInvariants}
Let $(M,g)$ be a locally homogeneous three-manifold with associated Ricci eigenvalues $\nu(g) = (\nu_1(g),\nu_2(g), \nu_3(g))$. Then, we may define 

\begin{eqnarray*}
b_0(M,g) &\equiv& a_0(M,g) \\
b_1(M,g) &=& P_1(\nu(g)) \\
b_2(M,g) &=& P_2(\nu(g))
\end{eqnarray*}
In the event that $(M,g)$ is locally isometric to a unimodular Lie group equipped with a left-invaraint metric such that $P_3(\nu(g)) \neq 0$ (i.e., it is positive), then we also define
$$b_3(M,g) = 30P_3(\nu(g)) + \frac{6P_1^2(\nu(g)) P_2^2(\nu(g)) - 24P_2^3(\nu(g))}{P_3(\nu(g))}.$$
\end{dfn}

The previous theorem shows us that the constants $b_j(M,g)$ $j = 0,1,2, 3$ form a collection of spectral invariants among locally homogeneous three-manifolds.
 
\begin{cor}\label{cor:HeatInvariantsSymmetricPolys}
Let $(M,g)$ and $(M',g')$ be locally homogeneous three-manifolds.
\begin{enumerate}
\item $a_j(M,g) = a_j(M',g')$ for $j = 0,1,2$ if and only if $b_j(M, g) = b_j(M', g')$ for $j = 0,1,2$.
\item Additionally, if $(M,g)$ and $(M',g')$ are each modeled on unimodular Lie groups outfitted with left-invariant metrics such that $P_3(\nu(g)), P_3(\nu(g')) \neq 0$, then $a_j(M,g) = a_j(M',g')$ for $j = 0,1,2, 3$ if and only if $b_j(M,g) = b_j(M',g')$ for $j = 0,1,2, 3$.
\end{enumerate}
\end{cor}

\begin{thm}\label{thm:P3}
Let $(M,g)$ (resp., $(M',g')$) be a compact locally homogeneous three-manifold locally isometric to a unimodular Lie group equipped with a left-invariant metric for which all eigenvalues of the Ricci tensor are non-zero. If $a_j(M,g) = a_j(M',g')$ for $j = 0,1,2,3$, then 
$$P_3(\nu(g')) = P_3(\nu(g))$$
or 
$$ P_3(\nu(g')) = C(M,g) \equiv \frac{6P_1^2(\nu(g))P_2^2(\nu(g)) - 24P_2^3(\nu(g)))}{30 P_3(\nu(g))}.$$
\end{thm}

\begin{proof}
The assumption concerning the heat invariants is equivalent to $b_j(M,g) = b_j(M',g')$ for $j = 0,1,2,3$. 
From which we deduce

$$30 P_3(\nu(g)) + \frac{6P_1^2(\nu(g))P_2^2(\nu(g)) - 24P_2^3(\nu(g)))}{P_3(\nu(g))} = 
30 P_3(\nu(g')) + \frac{6P_1^2(\nu(g))P_2^2(\nu(g)) - 24P_2^3(\nu(g)))}{P_3(\nu(g'))},$$
which has the claimed solutions.
\end{proof}

\section{Can You Hear Three-Dimensional Geometric Structures?}\label{sec:AudibleModelGeometries}

The goal of this section is to establish Theorem~\ref{thm:Main}. We will proceed in steps, but first we will collect a few useful observations.

\begin{lem}\label{lem:ObservationA}
Fix real numbers $\alpha, \beta$ and $\gamma$. Then, the multiset $[\alpha, \beta, \gamma]$ determines and is determined by the triple $(P_1(\alpha, \beta, \gamma), P_2(\alpha, \beta, \gamma), P_3(\alpha, \beta, \gamma))$, where $P_j$ is the $j$-th symmetric polynomial in three variables. 
\end{lem}

\begin{proof}
This follows immediately from the equation
$$(x+\alpha)(x+\beta)(x+\gamma) = x^3 + P_1(\alpha, \beta, \gamma)x^2 + P_2(\alpha, \beta, \gamma)x + P_3(\alpha, \beta, \gamma).$$
\end{proof}

\begin{lem}[\cite{MilnorLieGroups}, Corollary 4.4]\label{lem:ObservationB}
Let $(M,g)$ be a a three-dimensional manifold modeled on a unimodular Lie group equipped with a left-invariant metric. Then, $P_3(\nu(g)) \geq 0$. That is, the product of the Ricci eigenvalues of $(M,g)$ is nonnegative. Furthermore, equality happens if and only if at least two of the principal Ricci curvatures are zero.
\end{lem}

\begin{lem}\label{lem:ObservationC}
Let $(M,g)$ be a locally homogeneous three-manifold modeled on the geometry $(G,G)$, where $G$ is a simply-connected unimodular Lie group, and let $\nu(g) = (\nu_1(g), \nu_2(g), \nu_3(g))$ be the vector consisting of eigenvalues of the associated Ricci tensor $\Ric$.
\begin{enumerate}
\item If $P_3(\nu(g))$ is zero (i.e., $\Ric$ is degenerate), then $P_2(\nu(g))$ is zero.
\item $P_2(\nu(g))$ is non-negative if and only if one of the following holds: 
\begin{enumerate}
\item[(a)] $G= \widetilde{\operatorname{SL}_2(\R)}$ and $\Ric$ has signature $(0,0,-)$ or $(+,-,-)$ where $\nu_1(g) > 0 > \nu_2(g) \geq \nu_3(g)$ are such that $\nu_1(g) \leq -\frac{\nu_2(g) \nu_3(g)}{\nu_2(g) + \nu_3(g)}$;
\item[(b)] $G= S^3$ and $\Ric$ has signature $(+,+,+)$ or $(+,0,0)$; 
\item[(c)] $G = \operatorname{Sol}$ and $\Ric $ has signature $(0,0,-)$; or 
\item[(d)] $G = \R^3$ and $\Ric$ has signature $(0,0,0)$.
\end{enumerate}
\end{enumerate}
\end{lem}

\begin{proof}
The first statement follows immediately from Lemma~\ref{lem:ObservationB}. As for the second statement, it follows directly from Lemmas~\ref{lem:Ricci_sig} and \ref{lem:su_absolute}.
\end{proof}
  
We now recall the following observation of Milnor.

\begin{lem}[\cite{MilnorLieGroups}, Lemma 4.1]\label{lem:CrossProduct}
Let $G$ be a connected three-dimensional Lie group with Lie algebra $\germ{g}$ and equipped with a left-invariant metric $g$. Fix an orientation $\Omega$ on $\germ{g}$ and let $\times  : \germ{g} \times \germ{g} \to \mathfrak{g}$ be the cross-product determined by the inner product $\langle \cdot, \cdot \rangle = g_e(\cdot, \cdot)$. Then, the Lie bracket on $\germ{g}$ and the cross-product are related via the following formula
$$[u,v] = L(u \times v),$$
where $L: \germ{g} \to \germ{g}$ is a uniquely defined linear mapping. Furthermore, $G$ is unimodular if and only if $L$ is self-adjoint with respect to $\langle \cdot , \cdot \rangle$.

In the case where $G$ is unimodular, let $\{e_1, e_2, e_3\}$ be an orthonormal basis of $L$-eigenvectors with corresponding eigenvalues $\lambda_1, \lambda_2, \lambda_3$. Then, $\{e_1, e_2, e_3\}$ is also an orthonormal basis of $\Ric_g$-eigenvectors with corresponding eigenvalues $\nu_1 = 2\mu_2\mu_3$, $\nu_2 = 2\mu_1\mu_3$ and $\nu_3 = 2\mu_1\mu_2$, where
$$\left(\begin{array}{c}\mu_1 \\ \mu_2 \\ \mu_3\end{array}\right) = \frac{1}{2} 
\left(\begin{array}{ccc}-1 & 1 & 1 \\1 & -1 & 1 \\1 & 1 & -1\end{array}\right) \left(\begin{array}{c}\lambda_1 \\ \lambda_2 \\ \lambda_3 \end{array}\right).$$
Conversely, let $\mathcal{V}$ be the collection of $\nu \in \R^3$ for which---up to ordering---the signs of the components of $\nu$ are given by $(+,+,+)$, $(+,-,-)$, $(+,0,0)$, $(0,0,-)$ and $(0,0,0)$.\footnote{By Lemma~\ref{lem:Ricci_sig}, if $\nu$ is a triple consisting of Ricci-eigenvalues of a left-invariant metric on a three-dimensional unimodular Lie group, then $\nu$ is in $\mathcal{V}$.} Then, for any $\nu = (\nu_1, \nu_2, \nu_3) \in \mathcal{V}$  there is a unimodular Lie group $G$ and a left-invariant metric $g \in \mathscr{R}_{G}(G)$ such that the eigenvalues of the associated Ricci tensor $\Ric_g$ are given by $\nu_1, \nu_2$ and $\nu_3$.
\end{lem}

\begin{dfn}\label{dfn:MilnorEigenvalues}
Let $(\germ{g}, \Inner)$ be an oriented metric Lie algebra with corresponding cross-product $\times$. The unique linear transformation $L: \germ{g} \to \germ{g}$ such that $[u,v] = L( u \times v)$ is called the \emph{Milnor map}. In the case where $\germ{g}$ is unimodular, the eigenvalues of $L$ are referred to as the \emph{Milnor eigenvalues}.
\end{dfn}

From Lemma~\ref{lem:CrossProduct} we deduce the following.

\begin{lem}\label{lem:ObservationD}
Let $\germ{g}$ be the Lie algebra of a simply-connected three-dimensional unimodular Lie group $G$ and let $\Omega$ denote a fixed orientation on $\germ{g}$. Now, for $j = 1,2$, let $\langle \cdot , \cdot \rangle_j$ be an inner product on $\germ{g}$ with corresponding cross-product $\times_j$ and Milnor map $L_j: \germ{g} \to \germ{g}$.
Additionally, for $j = 1,2$, let $\{ e_{j1}, e_{j2}, e_{j3} \}$ be a  $\langle \cdot , \cdot \rangle_j$-orthonormal basis consisting of $L_j$-eigenvectors with $L(e_{jk}) = \lambda_{jk} e_{jk}$ for $k =1,2,3$. If $\lambda_{1k} = \lambda_{2k}$ for $k = 1,2,3$ or $\lambda_{1k} = -\lambda_{2k}$ for $k = 1,2,3$ (after relabeling, if necessary), then there is a Lie group automorphism $\Phi \in \operatorname{Aut}(G)$ such that $\Phi: (G, g_1) \to (G, g_2)$ is an isometry, where $g_j$ is the left-invariant metric induced by $\langle \cdot , \cdot \rangle_j$, for each $j =1,2$.  
\end{lem}

\begin{proof}
Let $\phi: \germ{g} \to \germ{g}$ be the linear map determined by $e_{1k} \mapsto e_{2k}$, for $k =1,2$. Since, we have 
\begin{eqnarray*}
\left[e_{j1}, e_{j2}\right] &=& L(e_{j3}) = \lambda_{j3} e_{j3} \\
\left[e_{j2}, e_{j3}\right] &=& L(e_{j1}) = \lambda_{j1} e_{j1} \\ 
\left[e_{j3}, e_{j1}\right] &=& L(e_{j2}) = \lambda_{j2} e_{j2}, 
\end{eqnarray*}
for $j=1,2$. We may conclude that $\phi$ is a Lie algebra isomorphism. One can check that the corresponding automorphism $\Phi \in \operatorname{Aut}(G)$ is an isometry between $(G,g_1)$ and $(G,g_2)$.
\end{proof}

We now establish that, on a three-dimensional simply-connected unimodular Lie group, isometry classes of left-invariant metrics possessing non-degenerate Ricci tensor can be mutually distinguished by the eigenvalues of their respective Ricci tensors.

\begin{prop}\label{prop:ObservationE}
Let $G$ be a simply-connected non-abelian unimodular Lie group of dimension three. And, for $j =1,2$, let $g_j$ be a  left-invariant metric on $G$ with non-degenerate Ricci tensor $\Ric_j$. Then, $\Ric_1$ and $\Ric_2$ have the same eigenvalues if and only if  $(G,g_1)$ and $(G,g_2)$ are isometric (via a Lie group automorphism). In the case  $G$ is  $\R^2 \rtimes \R$ we may remove the non-degeneracy condition; however, the flat left-invariant metrics on $G$ will not be isometric via a Lie group automorphism.
\end{prop}

\begin{rem}
The authors wish to thank Dorothee Schueth for suggesting how this might be proven.
\end{rem}

\begin{proof}
Let $g$ be a left-invariant metric on the group $G$ determined by the inner product $\langle \cdot, \cdot \rangle$ on $\germ{g}$. Fix an orientation on $\germ{g}$ and let $\times$ be the cross product determined by the inner product $\langle \cdot, \cdot \rangle = g_e$. Let $L: (\germ{g}, \langle \cdot, \cdot \rangle) \to (\germ(g), \langle \cdot, \cdot \rangle )$ be the self-adjoint map described in Lemma~\ref{lem:CrossProduct} and let $\{e_{1}, e_{2}, e_{3}\}$ be an orthonormal basis consisting of $L$-eigenvectors with corresponding eigenvalues $\lambda_1$, $\lambda_2$, and $\lambda_3$. Then, for $k=1,2,3$, $\Ric(e_{k}) = \nu_{k} e_k$, where 
$$\nu_{1} = 2\mu_{2}\mu_{3}, \;  \nu_{2} = 2\mu_{1}\mu_{3}, \; \mbox{ and } \nu_{3} = 2\mu_{1}\mu_{2}, $$
and  
$$\left(\begin{array}{c}\lambda_1 \\\lambda_2 \\ \lambda_3\end{array}\right) = 
\left(\begin{array}{ccc}0 & 1 & 1 \\1 & 0 & 1 \\1 & 1 & 0\end{array}\right) \left(\begin{array}{c}\mu_1 \\ \mu_2 \\ \mu_3 \end{array}\right).$$
\noindent
We will now see that (up to sign) the vector $\lambda = (\lambda_1, \lambda_2, \lambda_3)$ can be recovered from the vector $\nu = (\mu_1, \nu_2, \nu_3)$. The first statement of the corollary will then follow from a direct application of Lemma~\ref{lem:ObservationD}.

Indeed, by Lemma~\ref{lem:Ricci_sig}, the Ricci tensor of a unimodular Lie group is non-degenerate if and only if its signature is $(+, +, +)$ and $(+, -,-)$. The signature $(+,+,+)$ can only occur for the group $S^3$ in which case the entries of the vector $\mu = (\mu_1, \mu_2, \mu_3)$ are all positive or all negative. In any event, for any cyclic permutation $\sigma$ we have:
$$|\mu_{\sigma(1)}| = \frac{\nu_{\sigma(2)}\nu_{\sigma(3)}}{2\nu_{\sigma(1)}},$$
and we conclude that (up to sign) we may recover the vector $\mu$ from the Ricci eigenvalues. And, in turn, we may recover the vector $\lambda$ from $\mu$ (up to sign). On the other hand, consulting Lemma~\ref{lem:Ricci_sig} again, we see that every three-dimensional non-abelian unimodular Lie group  supports a left-invariant metric for which the Ricci tensor has signature $(+,-,-)$. In this case the signs of the entries of the vector $\mu = (\mu_1, \mu_2, \mu_3)$ must be $(+, -,-)$ or $(-, +, +)$. Once again, we see that for any cyclic permutation $\sigma$ we have
$$|\mu_{\sigma(1)}| = \frac{\nu_{\sigma(2)}\nu_{\sigma(3)}}{2\nu_{\sigma(1)}},$$
and we conclude that (up to sign) we may recover the vector $\mu$ from the Ricci eigenvalues. And, in turn, we may recover the vector $\lambda$ from $\mu$ (up to sign). This establishes the first statement.

In the case where $G$ is $\R^2 \rtimes \R$, the left-invariant metrics with degenerate Ricci tensor are all flat---which in dimension three is equivalent to being Ricci-flat---and are, therefore, isometric. This establishes the last statement.
 \end{proof}

\begin{cor}\label{cor:ObservationE}
Let $G$ be a simply-connected, non-abelian unimodular Lie group of dimension three and let $(M,g)$ be a locally homogeneous three-manifold modeled on the geometry $(G,G)$. Now, suppose either 
\begin{enumerate}
\item $G$ is not $S^3$, or 
\item $(M,g)$ has non-degenerate Ricci tensor.
\end{enumerate} 
Then, among closed locally homogeneous three-manifolds modeled on the geometry $(G,G)$, the first four heat invariants of $(M,g)$ determine the isometry class of the universal cover of $(M,g)$ up to two possibilities.
\end{cor}

\begin{proof}
By Lemma~\ref{lem:Ricci_sig}, the hypotheses ensure that $(M,g)$ has non-degenerate Ricci tensor. Indeed, when $G$ is not $S^3$, the Ricci tensor has signature $(+,-,-)$. Therefore, $P_3(\nu(g))$ is non-zero and, as in Theorem~\ref{thm:P3}, we may set 
$$C(M,g) \equiv \frac{6P_1^2(\nu(g))P_2^2(\nu(g)) - 24P_2^3(\nu(g))}{30P_3(\nu(g))}.$$
Now, suppose $(N,h)$ is another closed locally homogeneous three-manifold modeled on the geometry $(G,G)$ such that $a_j(M,g) = a_j(N,h)$ for $j =0,1,2$ and $3$. Then, it follows from Theorem~\ref{thm:HeatInvariantsSymmetricPolys} that $P_1(\nu(h)) = P_1(\nu(g))$ and $P_2(\nu(h)) = P_2(\nu(g))$.

In the case where $G$ is a three-dimensional simply-connected non-abelian unimodular Lie group that is not $S^3$, all spaces modeled on $(G,G)$ have non-degenerate Ricci tensor. Therefore, $P_3(\nu(g))$ and $P_3(\nu(h))$ are non-zero and, by Theorem~\ref{thm:P3}, we have $P_3(\nu(h)) = P_3(\nu(g))$ or $P_3(\nu(h)) = C(M,g)$. Then, applying Lemma~\ref{lem:ObservationA} and Proposition~\ref{prop:ObservationE}, the universal covering metric $\tilde{h}$ of $(N,h)$ belongs to one of (at most) two possible isometry classes.\footnote{In the case where $C(M,g) \leq 0$, the isometry class is unique.} 

Now, assuming $G$ is $S^3$, the fact that $(M,g)$ has non-degenerate Ricci tensor is equivalent to the Ricci tensor  having signature $(+,+,+)$ or $(+,-,-)$. Lemma~\ref{lem:su_absolute} ensures that in ether case we have $P_2(\nu(h)) = P_2(\nu(g))$ is non-zero. Therefore, by Lemma~\ref{lem:Ricci_sig}, the Ricci tensor of $(N,h)$ is also non-degenerate (i.e., has signature $(+,+,+)$ or $(+,-,-)$). This is equivalent to $P_3(\nu(g))$ and $P_3(\nu(h))$ being non-zero. We then run the same argument as in the case where $G$ is not $S^3$.
\end{proof}

\subsection{On the audibility of locally symmetric spaces of rank two}\label{sec:AudibilityNonConstCurvatureLocSymSpc} We will establish that among locally homogenoeus three-manifolds, compact Riemannian manifolds modeled on the $S^2 \times \R$-geometry (respectively, the $\mathbb{H}^2 \times \R$-geometry) are determined up to local isometry by their spectra.  Later, in Section~\ref{sec:S2TimesRGeometries}, we will prove the spaces modeled on the $S^2 \times \R$-geometry are actually uniquely determined by their spectra among locally homogeneous three-manifolds.

\begin{thm}\label{thm:AudibilityS2TimesR}
Let $(M,g)$ and $(M',g')$ be two locally homogeneous three-manifolds with the property that $a_j(M,g) = a_j(M',g')$, for $j = 0,1,2,3$. And, fix a non-zero real number $k$. Then,
$(M,g)$ is locally isometric to $\mathbb{S}^2_{k} \times \mathbb{E}$ (respectively, $\mathbb{H}^2_{-k} \times \mathbb{E}$) if and only if 
$(M',g')$ is locally isometric to $\mathbb{S}^2_{k} \times \mathbb{E}$ (respectively, $\mathbb{H}^2_{-k} \times \mathbb{E}$).
\end{thm}

\begin{proof}
Throughout we will let $\mathcal{M}(k)$ denote the simply-connected surface of dimension two. Now, fix $k \neq 0$ and assume that $(M,g)$ is locally isometric to $\mathcal{M}(k) \times \mathbb{E}$. Then, the principal curvatures are $K_{12} = k$ and $K_{13} = K_{23} = 0$, or equivalently the principal Ricci curvatures are $\nu_1 = \nu_2 = k$ and $\nu_3 = 0$. By Corollary~\ref{cor:HeatInvariantsSymmetricPolys}, the assumption that the first four heat invariants of $(M,g)$ and $(M',g')$ are identical is equivalent to: $a_0(M,g) = a_0(M',g')$, $P_1(\nu(g)) = P_1(\nu(g'))$, $P_2(\nu(g)) = P_2(\nu(g'))$ and $a_3(M,g) = a_3(M',g')$. So, we obtain

\begin{eqnarray*}
a_0(M,g) &=& a_0(M'g')\\
P_1(\nu(g)) &= &2k = P_1(\nu(g'))\\
P_2(\nu(g)) &= &k^2 = P_2(\nu(g'))\\
a_3(M,g) &=& a_3(M'g')\\
P_3(\nu(g)) &=& 0
\end{eqnarray*}

Lets first assume that $(M',g')$ is a locally symmetric three-manifold. Then, since three-manifolds of constant sectional curvature are uniquely determined up to local isometry by their first \emph{three} heat invariants \cite[Theorem 7.1]{Berger}, we see that $(M',g')$ is locally isometric to $\mathcal{M}(k') \times \mathbb{E}$ for some $k' \neq 0$. Then, 
$2k' =P_1(\nu(g')) = 2k$ and, therefore, $(M,g)$ and $(M', g')$ are locally isometric. What remains is to show that $(M',g')$ cannot be covered by a unimodular Lie group equipped with a left-invariant metric of (non-constant sectional cutvature).

Suppose, $(M',g')$ is not locally symmetric. Then, it must be locally isometric to a unimodular Lie group equipped with a left-invariant metric (of non-constant curvature). We first observe that in this case $P_3(\nu(g')) \neq 0$. Indeed, suppose $P_3(\nu(g')) =0$. Then, by Lemma~\ref{lem:ObservationC}, $P_2(\nu(g')) =0$, which contradicts our assumption on the heat invariants. Therefore, $P_3(\nu(g')) > 0$, since $P_3(\nu)$ is non-negative for any three-dimensional unimodular Lie group (Lemma~\ref{lem:ObservationB}).

Since $(M',g')$ is a locally isometric to a unimodular Lie group and $P_3(\nu(g')) > 0$, we see by Theorem~\ref{thm:HeatInvariantsSymmetricPolys} that 
$$a_3(M', g') = \frac{a_0(M',g')}{7!}(40k^3 - 30P_3(\nu(g'))).$$
On the other hand, 
$$a_3(M',g') = \frac{a_0(M,g)}{7!}64k^3.$$
Therefore, we find $P_3(\nu) = -\frac{24}{30}k^3$ (and conclude that $k <0$). 
Now, the Ricci eigenvalues $\nu_1(g')$, $\nu(g')$ and $\nu_3(g')$ are real roots of

\begin{eqnarray*}
0 &=& x^3 + P_1(\nu(g'))x^2 + P_2(\nu(g')) x + P_3(\nu(g'))\\
&=& x^3 + 2kx^2 + k^2 x -\frac{24}{30}k^3.
\end{eqnarray*}

\noindent
However, since the discriminant of this polynomial is negative, we conclude that it cannot have three real roots. Therefore, $(M',g')$ cannot be covered by a unimodular Lie group equipped with a left-invariant metric. 
\end{proof}

\begin{cor}\label{cor:AudibilityA}
Among compact locally homogeneous three-manifolds, locally symmetric spaces are determined up to local isometry by their first four heat invariants.
\end{cor}

\begin{proof}
This follows immediately from Theorem~\ref{thm:AudibilityS2TimesR} and \cite[Thm. 7.1]{Berger}.
\end{proof}

\subsection{On the audibility of three-dimensional nilmanifolds}\label{sec:AudibilityNilGeom}
There are countably infinite non-diffeomorphic three-dimensional manifolds admitting geometric structures modeled on $(\operatorname{Nil}, \operatorname{Nil})$, the $\operatorname{Nil}$-geometry \cite[Corollary 2.5]{GordonWilson86}. We will establish that among locally homogeneous three-manifolds, the property of being modeled on $(\operatorname{Nil}, \operatorname{Nil})$ is encoded in the spectrum. In fact, among locally homogeneous spaces, nilmanifolds are determined up to local isometry by their spectra. Coupling this result with a result of Gordon and Wilson, we will conclude that a nilmanifold is actually uniquely determined by its spectrum among all compact locally homogeneous three-manifolds. 

\begin{thm}\label{thm:AudibilityNilGeom}
Let $(M,g)$ and $(M',g')$ be two locally homogeneous three-manifolds with the property that $a_j(M,g) = a_j(M',g')$, for $j = 0,1,2,3$. If $(M,g)$ is modeled on $(\operatorname{Nil},\operatorname{Nil})$, then $(M',g')$ is locally isometric to $(M,g)$. 
\end{thm}

\begin{proof}
Since $(M,g)$ is modeled on $(\operatorname{Nil},\operatorname{Nil})$, its Ricci tensor has signature $(+,-,-)$ by Lemma \ref{lem:Ricci_sig}; in particular, it is non-degenerate. And, by Lemma~\ref{lem:su_absolute}, without loss of generality we may assume the eigenvalues of its Ricci tensor are given by $\nu_1 = |\nu_2| = |\nu_3| = c >0$. This coupled with our assumption on the heat invariants implies (via Corollary~\ref{cor:HeatInvariantsSymmetricPolys})

\begin{eqnarray*}
b_1(M,g) = P_1(\nu) &=& -c = P_1(\nu') \equiv b_1(M',g') \\
b_2(M,g) \equiv P_2(\nu) &=& -c^2 = P_2(\nu') = b_2(M',g')\\
P_3(\nu) &=& c^3
\end{eqnarray*}

Now, since locally symmetric spaces are determined up to local isometry by their first four heat invariants (see Corollary~\ref{cor:AudibilityA}), we see that $(M',g')$ must be modeled on a non-abelian unimodular Lie group equipped with a left-invariant metric (of non-constant curvature). Furthermore, since $P_2(\nu') =-c^2$ is non-zero, Lemma~\ref{lem:ObservationC} implies $P_3(\nu')$ must be non-zero. Therefore, by Theorem~\ref{thm:P3}, we see that 
$$P_3(\nu') = P_3(\nu)$$
or 
$$P_3(\nu') = \frac{6P_1^2(\nu)P_2^2(\nu) - 24P_2^3(\nu)}{30P_3(\nu)} = c^3 =P_3(\nu).$$
We now have $P_j(\nu') = P_j(\nu)$ for $j = 1,2,3$. Therefore, by Lemma~\ref{lem:ObservationA}, $\Ric$ and $\Ric'$ have the same eigenvalues and consequently both are of signature $(+,-,-)$. It follows from Lemma~\ref{lem:su_absolute} that $(M',g')$ is also modeled on $(\operatorname{Nil}, \operatorname{Nil})$ and by Proposition~\ref{prop:ObservationE} we conclude that $(M,g)$ and $(M',g')$ are locally isometric.
\end{proof}

We now establish that three-dimensional compact nilmanifolds are uniquely characterized by their spectra within the universe of locally homogeneous three-manifolds. 

\begin{proof}[Proof of Corollary~\ref{cor:MainNilGeometry}]
Gordon and Wilson have previously shown that in dimension three nilmanifolds can be mutually distinguished via their spectra \cite{GordonWilson86}. The result now follows by applying Theorem~\ref{thm:AudibilityNilGeom}. 
\end{proof}

\subsection{On the audibility of locally homogeneous platycosms}\label{sec:AudibilityR2SemiProdRGeom}
We recall that $(\R^2 \rtimes \R, \R^2 \rtimes \R)$, which we refer to as the $\R^2 \rtimes \R$-geometry, is a sub-geometry of $(\R^3, \Isom(\mathbb{E}^3))$. There are ten compact manifolds---sometimes referred to as ``platycosms''---that admit $\R^2 \rtimes \R$-geometries \cite{DoyleRossetti}, five of which are of the form $\Gamma \backslash (\R^2 \rtimes \R)$ for some co-compact discrete subgroup of $\R^2 \rtimes \R$ \cite{RaymondVasquez} and, therefore, admit non-flat structures. In particular, the three-torus can be realized in this manner. We show that, within the class of locally homogeneous three-manifolds, such a space is distinguished up to local isometry by its spectrum.

\begin{thm}\label{thm:AudibilityR2SemiProdRGeom}
Let $(M,g)$ and $(M',g')$ be two locally homogeneous three-manifolds with Ricci tensors $\Ric$ and $\Ric'$, respectively, and such that $a_j(M,g) = a_j(M',g')$, for $j = 0,1,2,3$. If $(M,g)$ is modeled on the $\R^2 \rtimes \R$-geometry, then $(M',g')$ is locally isometric to $(M,g)$.
\end{thm}

\begin{proof}
Since, as has been noted previously, closed three-manifolds of constant sectional curvature are determined up to local isometry by their first three heat invariants \cite[Theorem 7.1]{Berger}, we may assume that $(M,g)$ is a non-flat space modeled on the $\mathbb{R}^2 \rtimes \mathbb{R}$-geometry. Then, the signature of its Ricci tensor is $(+, -,-)$; in particular, it is non-degenerate. And, Lemma~\ref{lem:su_absolute} tells us the Ricci eigenvalues are given by $\nu_1 = |\nu_3| \equiv c >  |\nu_2| \equiv d >0$. Taking into account the assumption on the heat invariants, we then obtain

\begin{eqnarray*}
P_1(\nu) &=& -d = P_1(\nu')\\
P_2(\nu) &=& -c^2 = P_2(\nu')\\
P_3(\nu) &=& c^2d >0.
\end{eqnarray*}
\noindent

Now, by Corollary~\ref{cor:AudibilityA}, we know $(M',g')$ must be modeled on a unimodular Lie group equipped with a left-invariant metric (of non-constant curvature). Since $P_2(\nu')$ is non-zero, Lemma~\ref{lem:ObservationC} informs us that $P_3(\nu')$ is non-zero. Applying Theorem~\ref{thm:P3}, we find $P_3(\nu') = P_3(\nu)$ or 
$$P_3(\nu') = \frac{c^4d^2 - 4c^6}{5 c^2d} <0 .$$
The latter option cannot occur, because the product of the Ricci eigenvalues of a left-invariant metric on a three-dimensional unimodular Lie group must be non-negative (see Lemma~\ref{lem:ObservationB}). So, we have $P_j(\nu) = P_j(\nu')$ for $j = 1,2,3$.
Therefore, $\Ric$ and $\Ric'$ have the same eigenvalues and, therefore, signature $(+,-,-)$. By Lemma~\ref{lem:su_absolute}, we conclude that $(M',g')$ is a non-flat space modeled on the $\R^2 \rtimes \R$-geometry, and by Proposition~\ref{prop:ObservationE} we see that $(M,g)$ and $(M', g')$ are locally isometric.
\end{proof}

Theorems~\ref{thm:AudibilityS2TimesR}, \ref{thm:AudibilityNilGeom} and \ref{thm:AudibilityR2SemiProdRGeom} establish statement $(1)$ of Theorem~\ref{thm:Main} from which we may deduce the following.  

\begin{cor}\label{cor:AudibilityB}
Let $(M,g)$ be a compact three-manifold whose universal Riemannian cover is a symmetric space, $\operatorname{Nil}$ equipped with a left-invariant metric, or $\R^2 \rtimes \R$ equipped with a left-invariant metric. Then, among compact locally homogeneous three-manifolds, $(M,g)$ is determined up to local isometry by its first four heat invariants. 
\end{cor}

\begin{proof}
Follows directly from Corollary~\ref{cor:AudibilityA} and Theorems~\ref{thm:AudibilityS2TimesR}, \ref{thm:AudibilityNilGeom} and \ref{thm:AudibilityR2SemiProdRGeom}.
\end{proof}


\subsection{On the audibility of locally homogeneous elliptic three-manifolds}\label{sec:AudibilityS3Geom}
An \emph{elliptic} $n$-manifold is a manifold $\Gamma \backslash S^n$, where $\Gamma \leq  \operatorname{Diff}(S^n)$ acts freely and properly discontinuously. Up to diffeomorphism, an elliptic three-manifold is of the form $\Gamma \backslash S^3$, where $\Gamma \leq \SO(4)$ belongs to one of six infinite families of finite groups, and the locally homogeneous elliptic three-manifolds are precisely the Riemannian manifolds modeled on the metrically maximal geometry $(S^3, S^3)$. Our objective is to establish that the property of being a locally homogeneous elliptic three-manifold is audible among compact locally homogeneous three-manifolds and the signature of the Ricci tensor of such manifolds is spectrally determined. Furthermore, for certain left-invariant metrics $g_0$ on $S^3$ (e.g., constant curvature metrics), we find that for $g$ sufficiently close to $g_0$, the universal Riemannian cover $(S^3, g)$ is encoded in the spectra of its compact quotients.
 
\begin{thm}\label{thm:AudibilityS3Geom}
\textrm{}\\
\begin{enumerate}
\item Let $(M,g)$ and $(M',g')$ be two locally homogeneous three-manifolds with Ricci tensors $\Ric$ and $\Ric'$, respectively, and such that $a_j(M,g) = a_j(M',g')$, for $j = 0,1,2,3$. And, suppose further that $(M, g)$ is modeled on the $S^3$-geometry. Then, $(M',g')$ is also modeled on the $S^3$-geometry, and $\Ric$ and $\Ric'$ have the same signature.
In fact, if either (a) $\Ric$ has signature $(+, 0,0)$ or (b) $(M,g)$ has negative scalar curvature and $\Ric$ has signature $(+,-,-)$, then $\Ric$ and $\Ric'$ have the same eigenvalues. In particular, when $(M, g)$ has negative scalar curvature and $\Ric$ has signature $(+,-,-)$, then $(M,g)$ and $(M',g')$ are locally isometric.
\item Let $g_0$ be a left-invariant metric on $S^3$ with non-degenerate Ricci tensor and such that $C(S^3, g_0)$ (see Theorem~\ref{thm:P3} for definition) is negative. Then, within the space of left-invariant metrics on $S^3$, there is a neighborhood $\mathcal{U}$ of $g_0$ such that, among compact locally homogeneous three-manifolds, a space with universal Riemannian cover $(S^3, g)$, for some $g \in \mathcal{U}$, is determined up to local isometry by its first four heat invariants.  
\end{enumerate}

\end{thm}

\noindent
We have the following immediate corollary.

\begin{cor}\label{cor:AudibilityS3Geom}
\textrm{}\\
\begin{enumerate}
\item Let $(M,g)$ and $(M',g')$ be two locally homogeneous three-manifolds modeled on the $S^3$-geometry and for which $a_j(M,g) = a_j(M',g')$, for $j = 0,1,2,3$. Then, $(M,g)$ has non-degenerate Ricci tensor if and only if  $(M',g')$ has non-degenerate Ricci tensor.
\item Within the space of left-invariant metrics on $S^3$, there is a neighborhood $\mathcal{U}$ of the round metric such that, among locally homogeneous three-manifolds, a quotient of an $S^3$-geometry contained inside $\mathcal{U}$ is determined up to local isometry by its first four heat invariants.
\end{enumerate}
\end{cor}

As preparation for the proof of Theorem~\ref{thm:AudibilityS3Geom} we note that Lemma~\ref{lem:su_absolute} implies that, up to ordering, the set of all possible eigenvalues of the Ricci tensor of a left-invariant metric on $S^3$ with signature $(+,-,-)$ and negative scalar curvature is given by
$$\mathcal{S} = \{ (\alpha, \beta, \gamma) : \alpha > 0> \beta \geq \gamma, \; \alpha > |\gamma| \mbox{ and } \alpha + \beta + \gamma < 0 \}.$$
We then have the following fact that will be useful in our proof.

\begin{lem} \label{lem:Polysign}
The homogeneous symmetric polynomial $$f(\alpha, \beta, \gamma): =P_3^2-\frac{P_2^2(P_1^2-4P_2)}{5}$$
is nonpositive on $\mathcal{S}$, where $P_j \equiv P_j(\alpha, \beta, \gamma)$, for $j=1,2,3$.
\end{lem}

\begin{proof}
As the polynomial $f$ is symmetric and homogeneous, it suffices to show that $f$ is nonpositive on the set
$$\mathcal{S}_1=\{(1, \beta ,\gamma) :  0\geq \beta \geq \gamma \geq -1\, \text{ and } \beta+\gamma<-1\}.$$

To verify the nonpositivity of $f$ on the domain $\mathcal{S}_1$, we  do a change of variables:
\begin{equation}                                                                                                                                                                                                                                                                                                                                                                                                                                                                                                                                                                                                                                                                                                                                                                                                                                                                                                                                                                                                                                                                                                                                                                                                                                                                                                                                                                                                                                                                                                                                                                                                                                                                                                                                                                                                                                                                              
\begin{split}
x&:= \beta+\gamma\\
y&:= \beta \gamma.
\end{split}
\end{equation}
and show the nonpositivity of 
\begin{equation}
f(x,y)=y^2-\frac{1}{5}(x+y)^2((1+x)^2-4(x+y))
\end{equation}
on 
\begin{equation*} 
\mathcal{S}_2=\{(x, y) \,|\, -2\leq x\leq -1\, , 0\leq y \leq (1/4) x^2 \}.
\end{equation*}

First, note that
\begin{equation}\label{partial_f}
\frac{\partial f}{\partial y}=2y-\frac{2}{5}(x+y)(1+x^2-4x-4y)>0
\end{equation}
on $\mathcal{S}_2$. Therefore, the function $f$ has no critical point on the interior of $\mathcal{S}_2$. Now, we check the values of the function along the upper boundary curve
$$r(t)=(t, t^2/4), \text { } t\in [-2, -1].$$
Simple calculus shows that the function 
$g(t)=f(t, t^2/4)$ is indeed nonpositive on the interval. The partial derivative condition \eqref{partial_f} then implies that $f$ is nonpositive on $\mathcal{S}_2$, proving the Lemma.
\end{proof}

\begin{proof}[Proof of Theorem~\ref{thm:AudibilityS3Geom}]
\textrm{}\\

\begin{enumerate}
\item We begin by collecting some facts about the manifolds $(M,g)$ and $(M',g')$.
By applying Corollary~\ref{cor:AudibilityB} we may assume that $(M, g)$ is a space of non-constant sectional curvature modeled on the $S^3$-geometry. Then, Lemma~\ref{lem:Ricci_sig} tells us that $\Ric$ has signature $(+,+,+)$, $(+, 0,0)$ or $(+,-,-)$. Applying Corollary~\ref{cor:AudibilityB} once again, we see that  $(M', g')$ is a space of non-constant sectional curvature modeled on $(G,G)$, where $G$ is one of the unimodular Lie groups $S^3$, $\operatorname{Sol}$ or $\widetilde{\operatorname{SL}_2(\R)}$. We also observe that combining Theorem~\ref{thm:HeatInvariantsSymmetricPolys} with our assumption on the heat invariants implies $P_1(\nu(g)) = P_1(\nu(g'))$ and $P_2(\nu(g)) = P_2(\nu(g'))$.

Now, lets assume $\Ric$ has signature $(+,+,+)$. Then, $P_1(\nu(g))$, $P_2(\nu(g))$ and $P_3(\nu(g))$ are all positive. It follows that $\Scal(g') = P_1(\nu(g')) = P_1(\nu(g))$ is positive and, by Lemma~\ref{lem:Ricci_sig}, we conclude $(M',g')$ must be (a space of non-constant sectional curvature) modeled on the $S^3$-geometry. If $\Ric'$ were to have signature $(+,0,0)$, then we would have $P_2(\nu(g)) = P_2(\nu(g')) =0$, a contradiction. Similarly, if $\Ric'$ were to have signature $(+,-,-)$, then Lemma~\ref{lem:su_absolute}(1) implies $P_2(\nu(g)) = P_2(\nu(g'))$ is negative, which is also a contradiction. Therefore, we conclude that $\Ric'$ must also have signature $(+,+,+)$. Then, by Theorem~\ref{thm:P3}, we find $P_3(\nu')$ takes on at most two values. Therefore, by Lemma~\ref{lem:ObservationA} and Proposition~\ref{prop:ObservationE}, there are at most two isometry classes for the universal Riemannian cover of $(M',g')$.

Next, suppose that $(M,g)$ is such that $\Ric$ has signature $(+,0,0)$. Then, $P_1(\nu(g))$ is positive, while $P_2(\nu(g))$ and $P_3(\nu(g))$ are both zero. Then, once again, $\Scal(g') = P_1(\nu(g')) = P_1(\nu(g))$ is positive and, appealing to Lemma~\ref{lem:Ricci_sig}, we conclude that $(M',g')$ must be modeled on the $S^3$-geometry (of non-constant sectional curvature). Furthermore, since $P_2(\nu(g')) = P_2(\nu(g))$ is zero, we find the signature of $\Ric'$ must also be $(+,0,0)$: the signature $(+,-,-)$ is ruled out since Lemma~\ref{lem:su_absolute}(1) would imply $P_2(\nu(g'))$ is negative and the signature clearly cannot be $(+,+,+)$ as that would imply $P_2(\nu(g'))$ is positive. It then follows from the equality of $P_1(\nu(g'))$ and $P_1(\nu(g))$ that $\Ric$ and $\Ric'$ have the same eigenvalues and, as a result, identical signature.

Finally, suppose that $(M, g)$ is such that $\Ric$ has signature $(+,-,-)$. Then, $P_1(\nu(g))$ can have any sign, while $P_3(\nu(g))$ must be positive. As for $P_2(\nu(g))$, Lemma~\ref{lem:su_absolute}(1) implies $P_2(\nu(g))$ is negative. Since, $P_2(\nu(g')) = P_2(\nu(g)) < 0$ and $(M', g')$ is modeled on one of the unimodular Lie groups  $\operatorname{Sol}$, $\widetilde{\operatorname{SL}_2(\R)}$ or $S^3$ equipped with a left-invariant metric (of non-constant curvature), Lemma~\ref{lem:Ricci_sig} implies $\Ric'$ also has signature $(+,-,-)$.

To see that $(M', g')$ is modeled on an $S^3$-geometry, we first observe that if $P_1(\nu(g')) = P_1(\nu(g))$ is non-negative (i.e., both spaces are of non-negative scalar curvature), then Lemma~\ref{lem:Ricci_sig} implies $(M',g')$ must be modeled on an $S^3$-geometry (of non-constant sectional curvature).

Now, suppose $P_1(\nu(g')) = P_1(\nu(g))$ is negative (i.e., both spaces are of negative scalar curvature) and notice that, since $P_3(\nu(g))$ and $P_3(\nu(g'))$ are both non-zero, Theorem~\ref{thm:P3} implies $P_3(\nu(g)) = P_3(\nu(g'))$ or 
$$P_3(\nu(g')) = \frac{P_1^2(\nu(g))P_2(\nu(g)) - 4P_2^3(\nu(g))}{5P_3(\nu(g))}.$$
In the first case, we obtain $P_j(\nu(g)) = P_j(\nu(g'))$ for $j = 1,2,3$, and conclude by Lemma~\ref{lem:ObservationA} that $\Ric$ and $\Ric'$ have the same eigenvalues and, by Lemma~\ref{lem:su_absolute}, both are modeled on $S^3$. In the second case, since $\nu(g) = (\nu_1(g), \nu_2(g), \nu_3(g))$ is an element of the set $\mathcal{S}$, Lemma~\ref{lem:Polysign} informs us that 
$$P_3(\nu(g')) \geq P_3(\nu(g)).$$ Then, recalling that $P_j(\nu(g)) = P_j(\nu(g'))$ for $j = 1, 2$ and comparing the equations
$$0 = x^3 + P_1(\nu')x^2 + P_2(\nu')x + P_3(\nu') = (x+\nu_1')(x+\nu_2')(x+\nu_3')$$
and 
$$0 = x^3 + P_1(\nu)x^2 + P_2(\nu)x + P_3(\nu) = (x+\nu_1)(x+\nu_2)(x+\nu_3),$$
we determine that $\nu(g') = (\nu_1(g'), \nu_2(g'), \nu_3(g'))$ is also in $\mathcal{S}$. By Lemma~\ref{lem:su_absolute} and the definition of $\mathcal{S}$, we conclude that $(M',g')$ is modeled on an $S^3$-geometry for which $\Ric'$ also has signature $(+,-,-)$, but with \emph{possibly} different eigenvalues from those of $\Ric$. However, now that $\nu(g')$ is in $\mathcal{S}$ we may reverse the roles of $(M,g)$ and $(M',g')$ to obtain (via Lemma~\ref{lem:Polysign} and Theorem~\ref{thm:P3}) $P_3(\nu(g)) \geq P_3(\nu(g'))$. Hence, $P_j(\nu(g)) = P_j(\nu(g'))$ for $j = 1,2,3$ and, by Lemma~\ref{lem:ObservationA}, we conclude that $\Ric$ and $\Ric'$ have the same eigenvalues, all of which are non-zero. Therefore, applying Proposition~\ref{prop:ObservationE} $(M',g')$ and $(M,g)$ are locally isometric.

\item Throughout, we let $\mathscr{R}_{\rm{left}}(S^3)$ denote the space of left-invariant metrics on $S^3$ and recall that $C(M,g) \equiv \frac{6P_1^2(\nu(g))P_2^2(\nu(g)) - 24P_2^3(\nu(g))}{30P_3(\nu(g))}$, where $\nu(g)$ is the vector of eigenvalues of the Ricci tensor associated to $(M,g)$. Now, let $g \in \mathscr{R}_{\rm{left}}(S^3)$ have non-degenerate Ricci tensor (i.e., $P_3(\nu(g)) >0$) and be such that $C(S^3, g)$ is negative. Then, there is a neighborhood $\mathcal{U}$ of $g$ in $\mathscr{R}_{\rm{left}}(S^3)$ such that $C(S^3, h)$ is negative and $P_3(\nu(h))$ is positive for each $h \in \mathcal{U}$. Suppose $(M,h)$ is modeled on $(S^3, \tilde{h})$ for some $\tilde{h} \in \mathcal{U}$ and let $(M',h')$ be a compact locally homogeneous three-manifold such that $a_j(M',h') = a_j(M,h)$, for $j =0,1,2,3$. Then, by part (1) of the theorem, we know $(M',h')$ is locally isometric to an $S^3$-geometry and has non-degenerate Ricci tensor. Also, by Theorem~\ref{thm:HeatInvariantsSymmetricPolys}, $P_1(\nu(h')) = P_1(\nu(h))$ and $P_2(\nu(h')) = P_2(\nu(h))$. Since both $(M,h)$ and $(M',h')$ have non-degenrate Ricci tensor, Theorem~\ref{thm:P3} implies $P_3(\nu(h')) = P_3(\nu(h))$ or $P_3(\nu(h')) = C(M,h) <0$. The latter option contradicts Lemma~\ref{lem:ObservationB}, so we conclude $P_j(\nu(h')) = P_j(\nu(h))$ for $j =1,2,3$. Applying Lemma~\ref{lem:ObservationA} and Proposition~\ref{prop:ObservationE}, we conclude $(M,h)$ and $(M',h')$ are locally isometric. 

\end{enumerate}
 \end{proof}

\subsection{On the audibility of manifolds modeled on $(\operatorname{Sol}, \operatorname{Sol})$ and $(\widetilde{\operatorname{SL}_2(\R)}, \widetilde{\operatorname{SL}_2(\R)})$}\label{sec:AudibilitySolSL2R}
Theorems~\ref{thm:AudibilityS2TimesR}, \ref{thm:AudibilityNilGeom}, \ref{thm:AudibilityR2SemiProdRGeom} and \ref{thm:AudibilityS3Geom} establish the first and second statement of Theorem~\ref{thm:Main} from which we may deduce the following.

\begin{cor}\label{cor:AudibilityS3Geom}
Let $(M,g)$ and $(M',g')$ be two locally homogeneous three-manifolds with the property that $a_j(M,g) = a_j(M',g')$, for $j = 0,1,2,3$. Then, $(M,g)$ is modeled on $(\operatorname{Sol}, \operatorname{Sol})$ or $(\widetilde{\operatorname{SL}_2(\R)}, \widetilde{\operatorname{SL}_2(\R)})$ if and only if $(M',g')$ is modeled on $(\operatorname{Sol}, \operatorname{Sol})$ or $(\widetilde{\operatorname{SL}_2(\R)}, \widetilde{\operatorname{SL}_2(\R)})$.
\end{cor}

\begin{proof}
This follows directly from \cite[Thm. 7.1]{Berger} and Theorems~\ref{thm:AudibilityS2TimesR}, \ref{thm:AudibilityNilGeom}, \ref{thm:AudibilityR2SemiProdRGeom} and \ref{thm:AudibilityS3Geom}.
\end{proof}

Regarding isospectral pairs that are modeled on the $\widetilde{\operatorname{SL}_2(\R)}$-geometry, by using  arguments similar to those in the previous sections, we find that certain metrics with non-degenerate Ricci tensor are determined up to local isometry by their spectra.

\begin{prop}\label{prop:SL2RNonDegenerate}
Let $(M,g)$ be a compact locally homogeneous three-manifold with Ricci tensor $\Ric$ and let $\nu(g) = (\nu_1(g), \nu_2(g), \nu_3(g))$ denote the vector of $\Ric$-eigenvalues. 
Now, suppose $(M,g)$ is modeled on the metrically maximal geometry $(\widetilde{\operatorname{SL}_2(\R)}, \widetilde{\operatorname{SL}_2(\R)})$ and is such that $\Ric$ has signature $(+,-,-)$ and $\nu_1(g) \leq -\frac{\nu_2(g)\nu_3(g)}{\nu_2(g) + \nu_3(g)}$ (after possibly rearranging the eigenvalues). If $(M',g')$ is a compact locally homogeneous three-manifold satisfying  $a_j(M,g) = a_j(M',g')$ for $j = 0,1,2,3$, then $(M',g')$ is also modeled on the metrically maximal geometry $(\widetilde{\operatorname{SL}_2(\R)}, \widetilde{\operatorname{SL}_2(\R)})$ and is such that $\Ric'$ has signature $(+,-,-)$ and $\nu_1(g') \leq -\frac{\nu_2(g') \nu_3(g')}{\nu_2(g') + \nu_3(g')}$ (after possibly rearranging the eigenvalues). Furthermore, if $P_1^2(\nu(g)) - 4P_2(\nu(g))$ is negative, then $(M, g)$ is determined up to local isometry by its first four heat invariants.
\end{prop}

\begin{proof}
By Corollary~\ref{cor:AudibilityS3Geom}, $(M',g')$ must be modeled on $(\operatorname{Sol}, \operatorname{Sol})$ or $(\widetilde{\operatorname{SL}_2(\R)}, \widetilde{\operatorname{SL}_2(\R)})$. And, by Lemma~\ref{lem:ObservationC}(2), $P_2(\nu(g))$ is positive. Now, using Theorem~\ref{thm:HeatInvariantsSymmetricPolys}, the hypothesis on the heat invariants implies $P_1(\nu(g') =P_1(\nu(g))$ and $P_2(\nu(g')) = P_2(\nu(g))$. Then, the first conclusion of the proposition is reached  by Lemma~\ref{lem:ObservationC}(2). To establish the last statement, we apply Lemma~\ref{lem:ObservationB} and Theorem~\ref{thm:P3}.   
\end{proof}

We also remark that among compact Riemannian three-manifolds modeled on the $\operatorname{Sol}$-geometry the spectrum encodes local geometry.

\begin{prop} \label{prop:Solgeometry}
Let $(M_1,g_1)$ and $(M_2,g_2)$ be two locally homogeneous three-manifolds with Ricci tensors $\Ric_1$ and $\Ric_2$, respectively, and such that $a_j(M_1,g_1) = a_j(M_2,g_2)$, for $j = 0,1,2$. If $(M_1,g_1)$ and $(M_2, g_2)$ are both modeled on the metrically maxial geometry $(\operatorname{Sol}, \operatorname{Sol})$, then $(M_1,g_1)$ is locally isometric to $(M_2,g_2)$.
\end{prop}

\begin{proof}
As $(M_1, g_1)$ and $(M_2, g_2)$ are both modeled on the $\operatorname{Sol}$-geometry, we may assume that there exists nonzero constants $a_i$ and $b_i$ such that the eigenvalues of the Milnor map $L_i$ are given by $\lambda_{1, i}=2a_i^2$, $\lambda_{2, i}=-2b_i^2$, and $\lambda_{3, i}=0$, for all $i=1, 2$. Taking into account the assumption on the heat invariants, Theorem~\ref{thm:HeatInvariantsSymmetricPolys} yields: 
\begin{equation} \label {Sol:a_1}
P_1(\nu(g_1)) =-2(a_1^2+b_1^2)^2=-2(a_2^2+b_2^2)^2= P_1(\nu(g_2))
\end{equation}
\begin{equation} \label{Sol:a_2}
P_2(\nu(g_1)) = -4(a_1^2+b_1^2)^2(a_1^2-b_1^2)^2 = -4(a_2^2+b_2^2)^2(a_2^2-b_2^2)^2= P_2(\nu(g_2)).
\end{equation}

Suppose that $a_1^2=b_1^2$. Then, by equation \eqref{Sol:a_2}, $a_2^2=b_2^2$. The equation \eqref{Sol:a_1} then implies that 
$$a_1^2=a_2^2=b_1^2=b_2^2.$$
Hence, the eigenvalues of the Milnor map for $(M_1, g_1)$ and $(M_2, g_2)$ are equal. By Lemma~\ref{lem:ObservationD}, $(M_1, g_1)$ and $(M_2, g_2)$ are locally isometric.

When $a_1^2\neq b_1^2$, one can again check that the eigenvalues of the Milnor map for $(M_1, g_1)$ and $(M_2, g_2)$ are equal, and the proposition again follows from Lemma~\ref{lem:ObservationD}.
\end{proof}

\subsection{The proof of the main theorem} We now prove Theorem~\ref{thm:Main}.

\begin{proof}[Proof of Theorem~\ref{thm:Main}] This follows by combining  Theorems~\ref{thm:AudibilityS2TimesR}, \ref{thm:AudibilityNilGeom}, \ref{thm:AudibilityR2SemiProdRGeom} and \ref{thm:AudibilityS3Geom}, 
and Propositions~\ref{prop:SL2RNonDegenerate} and ~\ref{prop:Solgeometry}.
\end{proof}

\section{Distinguishing Manifolds Modeled on the $S^2 \times \R$-Geometry}\label{sec:S2TimesRGeometries}

The goal of this section is to establish Corollary~\ref{cor:MainS2TimesR} which states that, among locally homogeneous three-manifolds, a three-manifold modeled on  $(S^2 \times \R, \Isom(\mathbb{S}^2 \times \mathbb{E})^0)$ is uniquely determined by its spectrum. The result follows immediately from the following proposition.

\begin{prop}\label{prop:S2TimesRSpectra}
Fix $k >0$. Isospectral compact locally symmetric spaces locally isometric to $\mathbb{S}_{k}^2 \times \mathbb{E}$ are isometric.  
\end{prop}
\vskip 5pt

Before proving Proposition~\ref{prop:S2TimesRSpectra} we provide an argument for Corollary~\ref{cor:MainS2TimesR}.

\begin{proof}[Proof of Corollary~\ref{cor:MainS2TimesR}]
Let $(M,g)$ be a compact locally symmetric three-manifold whose Riemannian universal cover is $\mathbb{S}_{k}^{2} \times \mathbb{E}$. If $(M',g')$ is a comapct locally homogeneous three-manifold isospectral to $(M,g)$, then Theorem~\ref{thm:Main} implise that $(M',g')$ is also locally isometric to $\mathbb{S}_{k}^{2} \times \mathbb{E}$. Then, by Proposition~\ref{prop:S2TimesRSpectra}, we see $(M,g)$ and $(M', g')$ are isometric.
\end{proof}

In order to prove Proposition~\ref{prop:S2TimesRSpectra}, we first describe the compact quotients of $\mathbb{S}_{k}^2 \times \mathbb{E}$ up to isometry (cf. \cite{Scott}).  Given a real number $v$, define isometries $\tau_v, R_v \in \Isom(\mathbb{R})$ by $\tau_v(x)=x+v$ and $R_v(x)=2v-x$ for each $x \in \mathbb{R}$.  Geometrically, $\tau_v$ is a translation by $v$ and $R_v$ is a reflection fixing $v$. Let $v \in \mathbb{R}$ be positive. For each integer  $1\leq i \leq 4$, define subgroups $\Gamma_i(v)$ of $\Isom(\mathbb{S}^2\times \mathbb{E})=\Isom(\mathbb{S}^2) \times \Isom(\mathbb{E})$ as follows:  $$\Gamma_1(v)= \langle (I, \tau_v) \rangle\,\,\,\,\,\,\,\,\,\,\,\Gamma_2(v)= \langle (-I,\tau_v) \rangle$$ $$\Gamma_3(v)=\langle (-I,R_0),(-I,R_v) \rangle\,\,\,\,\,\,\,\,\,\,\, \Gamma_4(v)=\langle (-I,I), (I, \tau_v)\rangle.$$  The groups $\Gamma_i(v)$ act isometrically, properly discontinuously, and freely on $X:=\mathbb{S}_k^2 \times \mathbb{E}$. Let $M_i(k,v)$ denote the compact locally symmetric manifold defined by $M_i(k,v):=\Gamma_i(v) \backslash X$.
Up to diffeomorphism, one can see that $M_1(k,v)$ is $S^2 \times S^1$, $M_2(k,v)$ is the non-trivial $S^1$-bundle over $\mathbb{R}P^2$, $M_3(k,v)$ is $\mathbb{R}P^2 \# \mathbb{R}P^2$ and $M_4(k,v)$ is $\mathbb{R}P^2 \times S^1$. We omit the proof of the following well-known proposition (c.f. \cite{Scott}).

\begin{prop}\label{prop:prepprop}
If $(M,g)$ is a compact locally symmetric space with universal Riemannian covering $\mathbb{S}_k^2 \times \mathbb{E}$, then there exists a unique positive real number $v$ and a unique integer $1\leq i \leq 4$ such that $(M,g)$ is isometric to $M_i(k,v)$. 
\end{prop}

\begin{lem}\label{volume1}
The volumes of compact locally symmetric spaces with universal Riemannian covering $\mathbb{S}^2_{k} \times\mathbb{E}$ are given by $$\vol(M_1(k,v))=\vol(M_2(k,v))=\vol(M_3(k,v))=2\vol(M_4(k,v))=\frac{4\pi v}{k}.$$
\end{lem}

\begin{proof}
The set $S^2 \times [0,v)$ is a fundamental domain for the actions of $\Gamma_1(v)$ and $\Gamma_2(v)$ on $S^2 \times \mathbb{R}$.  Therefore $\vol(M_1(k,v))=\vol(M_2(k,v))=\vol(\mathbb{S}_k^2 \times [0,v))=\frac{4\pi v}{k}$. The group $\Gamma_1(v)$ is an index two subgroup of $\Gamma_4(v)$.  Therefore $M_1(k,v)$ double covers $M_4(k,v)$ whence $2\vol(M_4(k,v)) = \vol(M_1(k,v))$. Note that since $R_v\circ R_0=\tau_{2v}$, $$\Gamma_3(v)=\langle (-I,R_0), (-I, R_v) \rangle=\langle (-I,R_0),(I,\tau_{2v}) \rangle.$$ It follows that $\Gamma_1(2v)$ is an index two subgroup of $\Gamma_3(v)$ whence $2\vol(M_3(k, v)) = \vol(M_1(k, 2v))=\frac{8\pi v}{k}$, concluding the proof.
\end{proof}

For each positive real number $v$ and integer $1\leq i \leq 4$, let $\mathscr{E}_i(k,v)$ denote the set of eigenvalues of the Laplace-Beltrami operator associated to $M_i(k,v)$. These eigenvalue sets are characterized in the next Lemma.  The proof is based on a few well-known facts that we now describe (c.f. \cite{Chavel, BGM}).\vskip 5pt 

\noindent \textit{Fact 1:} If $\pi: X \rightarrow M$ is a Riemannian covering, then $\lambda \in \mathbb{R}$ is an eigenvalue for $M$ if and only if $\lambda$ is an eigenvalue for $X$ whose eigenspace contains eigenfunctions invariant under the deck group of $\pi$.\vskip 5pt

\noindent \textit{Fact 2:} If $M \times N$ is a Riemannian product, then eigenfunctions for $M \times N$ with eigenvalue $\lambda$ are linear combinations of products of eigenfunctions for $M$ and $N$ whose eigenvalues sum to $\lambda$. \vskip 5pt 

\noindent \textit{Fact 3:} The set of eigenvalues for $\mathbb{S}_k^2$ is given by $\{m(m+1)k\, \vert\, m \in \mathbb{Z}_{\geq 0}\}$.  The eigenfunctions corresponding to an eigenvalue $m(m+1)$ are the restrictions to $S^2$ of harmonic homogeneous degree $m$ polynomial functions on $\mathbb{R}^3$. \vskip 5pt

\noindent \textit{Fact 4:}  For each $\lambda \in \mathbb{R}$, the nonnegative real number $\lambda^2 \in \mathbb{R}$ is an eigenvalue for $\mathbb{E}$ with corresponding eigenspace $E_{\lambda^2}:=\{a\cos(\lambda t)+b\sin(\lambda t)\, \vert\, (a,b) \in \mathbb{R}^2\}$.

\begin{lem}\label{calculate}
For fixed $k, v>0$, define $F:\mathbb{Z}_{\geq 0}\times \mathbb{Z}_{\geq 0} \rightarrow \mathbb{R}$ by $F(m,n)=m(m+1)k+\pi^2v^{-2}n^2$.  The set of eigenvalues of the Laplace-Beltrami operator for a compact locally symmetric space with universal Riemannian covering $\mathbb{S}^2_k \times \mathbb{E}$ are given by

\begin{enumerate}
\item  $\mathscr{E}_1(k,v)=\{F(m,n)\, \vert\, (m,n)\in \mathbb{Z}_{\geq0} \times \mathbb{Z}_{\geq0}\,\, \text{and}\,\,\, n\equiv 0 \mod 2\}$
\item  $\mathscr{E}_2(k,v)=\{F(m,n)\, \vert\,(m,n)\in \mathbb{Z}_{\geq0} \times \mathbb{Z}_{\geq0}\,\,\,\text{and}\,\,\, m\equiv n \mod 2 \}$
\item $\mathscr{E}_3(k,v)=\{F(m,n)\, \vert\, (m,n) \in \mathbb{Z}_{\geq0} \times \mathbb{Z}_{\geq0}\}$
\item $\mathscr{E}_4(k,v)=\{F(m,n)\, \vert\, (m,n)\in \mathbb{Z}_{\geq0} \times \mathbb{Z}_{\geq0} \,\,\, \text{and}\,\,\, m\equiv n\equiv 0 \mod 2 \}$.
\end{enumerate}
\end{lem}

\begin{proof}
By Facts 1-4, for each integer $1\leq i\leq 4$, $\mu \in \mathscr{E}_i(k,v)$ if and only if there exist $m\in \mathbb{Z}_{}$, $\lambda \in \mathbb{R}$, a harmonic homogeneous degree $m$ polynomial function $\phi(x)$, and a function $f(t) \in E_{\lambda^2}$ such that  $\mu=m(m+1)k+\lambda^2$ and such that the function  $G(x,t):=\phi(x)\cdot f(t)$ is invariant under the action of $\Gamma_i(v)$. 

Use Fact 4 to verify that a function $f \in E_{\lambda^2}$ satisfies $f\circ \tau_v=f$ (respectively, $f\circ \tau_f=-f$) 
if and only if there exists an \textit{even} (respectively, \textit{odd}) integer $n \in \mathbb{Z}$ such that $\lambda^2=\pi^2v^{-2}n^2.$  

The four eigenvalue sets are now determined by the following invariance requirements of a function of the form $G(x,t)=\phi(x)\cdot f(t)$ as described above.\vskip 3pt

\noindent \textit{Invariance under} $\Gamma_1(v)$: The function $G(x,t)$ is $\Gamma_1(v)$ invariant if and only if $f\circ \tau_v=f$. \vskip 3pt

\noindent \textit{Invariance under} $\Gamma_2(v)$: The function $G(x,t)$ is $\Gamma_2(v)$ invariant if and only if 1) the degree of $\phi$ is even and $f \circ \tau_v=f$ or  2) the degree of $\phi$ is odd and $f \circ \tau=-f$. \vskip 3pt

\noindent \textit{Invariance under} $\Gamma_3(v)$: Recall that $\Gamma_3(v)=\langle (-I,R_0),(I,\tau_{2v}) \rangle$.  Hence, the function $G(x,t)$ is $\Gamma_3(v)$ invariant if and only if  $\phi(x)\cdot f(t)=\phi(-x)\cdot f(-t)$ and $f\circ \tau_{2v}=f$.  The former equality holds provided that $f(t)$ is a multiple of $\cos(\lambda t)$ (respectively, $\sin(\lambda t)$) when $\phi(x)$ has even (respectively, odd) degree.  The latter equality holds provided that there exists an even integer $j$ such that $\lambda^2=\pi^2(2v)^{-2}j^2$, or equivalently, there exists an integer $n$ such that $\lambda^2=\pi^2v^{-2}n^2$.\vskip 3pt

\noindent \textit{Invariance under} $\Gamma_4(v):$  The function $G(x,t)$ is $\Gamma_4(v)$ invariant provided that $\phi(x)\cdot f(t)=\phi(-x)\cdot f(t)$ and $f\circ \tau_v=f$.  As above, the former (respectively, latter) equality holds if and only if  $m$ is even (respectively, $n$ is even).
\end{proof}
\vskip 5pt

\noindent \textit{Proof of Proposition~\ref{prop:S2TimesRSpectra}.} Isospectral manifolds have equal volumes and, by Lemma~\ref{volume1}, the compact quotients of $\mathbb{S}_k^2\times \mathbb{E}$ of equal volume are $M_1(k,v)$, $M_2(k,v)$, $M_3(k,v)$ and $M_4(k, 2v)$. Therefore, it suffices to prove that for each pair of positive real numbers $k, v \in \mathbb{R}$, the sets $\mathscr{E}_1(k,v)$, $\mathscr{E}_2(k, v)$, $\mathscr{E}_3(k,v)$, and $\mathscr{E}_4(k, 2v)$ are mutually distinct. 

Without loss of generality, we can rescale the metrics so that $k =1$ and use Lemma~\ref{calculate} to deduce
\begin{eqnarray*}
\mathscr{E}_1(1,v) &=& \{F(m,n)\vert\, (m,n)\in \mathbb{Z}_{\geq0} \times \mathbb{Z}_{\geq0}\,\,\text{and}\,\, n\equiv 0 \mod 2\}\\  
\mathscr{E}_2(1, v) &=& \{F(m,n)\, \vert\,(m,n)\in \mathbb{Z}_{\geq0} \times \mathbb{Z}_{\geq0}\,\,\text{and}\,\, m\equiv n\mod 2 \}\\
\mathscr{E}_3(1, v) &=& \{F(m,n)\, \vert\, (m,n) \in \mathbb{Z}_{\geq0} \times \mathbb{Z}_{\geq0}\}\\
\mathscr{E}_4(1, 2v) &=& \{F(m,n)\, \vert\, (m,n)\in \mathbb{Z}_{\geq0} \times \mathbb{Z}_{\geq0} \,\,\text{and}\,\, m\equiv 0 \mod 2 \}.
\end{eqnarray*}
\noindent
As $F(m,n)$ is monotonically increasing in $m$ and $n$, the smallest positive element of each of these sets belong to the following subsets: 
\begin{eqnarray*}
\{F(1,0),F(0,2)\} &=&\{2,4\pi^2v^{-2}\}\subset \Lambda_1(1, v)\\
 \{F(0,2),F(1,1),F(2,0)\} &=& \{4\pi^2v^{-2}, 2+\pi^2v^{-2},6\}\subset \Lambda_2(1, v)\\
 \{F(0,1),F(1,0)\} &=& \{\pi^2v^{-2},2\}\subset \Lambda_3(1, v)\\
 \{F(0,1),F(2,0)\} &=& \{\pi^2v^{-2},6\}\subset \Lambda_4(1, 2v).
 \end{eqnarray*} 
 \noindent
 We complete the proof in the following six steps. \vskip 5pt

\noindent 1. \textit{Showing} $\mathscr{E}_1(1,v)\neq \mathscr{E}_2(1,v):$  If $2=4\pi^2v^{-2}$, then $F(m,n)=m(m+1)+n^2/2$. In this case, $F(1,1)=5/2$ is a member of $\mathscr{E}_2(1,v)$ but not of $\mathscr{E}_1(1,v)$.  If $2\neq 4\pi^2v^{-2}$, use the subsets above to conclude that either the  smallest or second smallest positive elements in these sets differ.\vskip 3pt
\noindent 2. \textit{Showing} $\mathscr{E}_1(1, v)\neq \mathscr{E}_3(1, v):$ Use the subsets above to conclude that either the smallest or second smallest positive elements in these sets differ.\vskip 3pt
\noindent 3. \textit{Showing} $\mathscr{E}_1(1, v)\neq \mathscr{E}_4(1, 2v):$ If $2=\pi^2v^{-2}$, then $F(m,n)=m(m+1)+2n$.  In this case, $F(4,1)=22$ is a member of $\mathscr{E}_4(1,2v)$ but not of $\mathscr{E}_1(1,v)$.  If $2\neq \pi^2v^{-2}$, use the subsets above to conclude that the smallest positive elements in these sets differ.\vskip 3pt
\noindent 4. \textit{Showing} $\mathscr{E}_2(1, v)\neq \mathscr{E}_3(1, v):$ Use the subsets above to conclude that the smallest positive elements in these sets differ. \vskip 3pt
\noindent 5. \textit{Showing} $\mathscr{E}_2(1, v) \neq \mathscr{E}_4(1, 2v):$ If $6=\pi^2v^{-2}$, then $F(m,n)=m(m+1)+6n^2$.  In this case, $F(5,1)=36$ is a member of $\mathscr{E}_2(1,v)$ but not of $\mathscr{E}_4(1, 2v)$ .  If $6\neq \pi^2v^{-2}$, use the subsets above to conclude that the smallest positive elements in these sets differ. \vskip 3pt
\noindent 6. \textit{Showing} $\mathscr{E}_3(1, v) \neq \mathscr{E}_4(1, 2v):$ If $2=\pi^2v^{-2}$, then $F(m,n)=m(m+1)+2n^2$.  In this case, $F(1,2)=10$ is a member of $\mathscr{E}_3(1, v)$ but not of $\mathscr{E}_4(1, 2v)$ . If $2 \neq \pi^2v^{-2}$, use the subsets above to conclude that the smallest positive elements in these sets differ.
\qed
\vskip 5pt


\end{document}